\documentclass[reqno]{amsart}

\usepackage{enumitem} 
\usepackage{bbm}
\usepackage{hyperref}
\newcommand*{\mailto}[1]{\href{mailto:#1}{\nolinkurl{#1}}}
\newcommand{\arxiv}[1]{\href{http://arxiv.org/abs/#1}{arXiv:#1}}

\newtheorem{theorem}{Theorem}[section]
\newtheorem{definition}[theorem]{Definition}
\newtheorem{lemma}[theorem]{Lemma}
\newtheorem{proposition}[theorem]{Proposition}
\newtheorem{corollary}[theorem]{Corollary}
\newtheorem{remark}[theorem]{Remark}

\theoremstyle{definition}
\newtheorem{example}{Example}

\newcommand{\R}{{\mathbb R}}
\newcommand{\N}{{\mathbb N}}
\newcommand{\Z}{{\mathbb Z}}
\newcommand{\C}{{\mathbb C}}
\newcommand{\U}{U}

\newcommand{\oo}{o}

\newcommand{\ti}{\tilde}
\newcommand{\be}{\begin{equation}}
\newcommand{\ee}{\end{equation}}
\newcommand{\ba}{\begin{array}}
\newcommand{\ea}{\end{array}}

\newcommand{\x}{\mathrm{x}}

\newcommand{\id}{\mathbbm{1}}
\newcommand{\E}{\mathrm{e}}
\newcommand{\I}{\mathrm{i}}

\newcommand{\tr}{\mathrm{tr}}

\newcommand{\loc}{\mathrm{loc}}
\newcommand{\re}{\mathrm{Re}}

\DeclareMathOperator{\ran}{ran}

\newcommand{\floor}[1]{\lfloor#1 \rfloor}

\newcommand{\ledot}{\,\cdot\,}
\newcommand{\redot}{\cdot\,}

\newcommand{\dip}{\upsilon}
\newcommand{\Wr}{\mathsf{w}}

\newcommand{\sig}{\sigma}

\numberwithin{equation}{section}

\begin{document}

\title[Spectral asymptotics]{Spectral asymptotics for canonical systems}

\author[J.\ Eckhardt]{Jonathan Eckhardt}
\address{School of Computer Science \& Informatics\\ Cardiff University\\ Queen's Buildings \\ 
5 The Parade\\ Roath \\ Cardiff CF24 3AA\\ Wales \\ UK}
\email{\mailto{EckhardtJ@cardiff.ac.uk}}

\author[A.\ Kostenko]{Aleksey Kostenko}
\address{Faculty of Mathematics\\ University of Vienna\\
Oskar-Morgenstern-Platz 1\\ 1090 Wien\\ Austria}
\email{\mailto{duzer80@gmail.com};\mailto{Oleksiy.Kostenko@univie.ac.at}}
\urladdr{\url{http://www.mat.univie.ac.at/~kostenko/}}

\author[G.\ Teschl]{Gerald Teschl}
\address{Faculty of Mathematics\\ University of Vienna\\
Oskar-Morgenstern-Platz 1\\ 1090 Wien\\ Austria\\ and International
Erwin Schr\"odinger
Institute for Mathematical Physics\\ Boltzmanngasse 9\\ 1090 Wien\\ Austria}
\email{\mailto{Gerald.Teschl@univie.ac.at}}
\urladdr{\url{http://www.mat.univie.ac.at/~gerald/}}

\thanks{J.\ Reine Angew.\ Math.\ {\bf 736}, 285--315 (2018)}
\thanks{{\it Research supported by the Austrian Science Fund (FWF) under Grants No.\ J3455, P26060, and Y330}}


\begin{abstract}
Based on continuity properties of the de Branges correspondence, we develop a new approach to study the high-energy behavior of Weyl--Titch\-marsh and spectral functions of $2\times2$ first order canonical systems. Our results improve several
classical results and solve open problems posed by previous authors. Furthermore, they are applied to radial Dirac and radial Schr\"odinger operators as well as to Krein strings and generalized indefinite strings. 
\end{abstract}

\maketitle

\section{Introduction}
\label{sec:int}

The $m$-function or Weyl--Titchmarsh function, introduced by H.\ Weyl in \cite{weyl} and given much more prominence by E.\ C.\ Titchmarsh \cite{tit} and K.\ Kodaira \cite{ko}, plays a fundamental role in the spectral theory for Sturm--Liouville operators. In particular, it is known that all information about spectral properties of the operator 
\be
 -\frac{d^2}{dx^2} + q(x),
\ee
 acting on $L^2([0,\infty))$, is encoded in this function. In 1952, V.\ A.\ Marchenko proved (see \cite[Theorem 2.2.1]{mar52}) that the $m$-function corresponding to the Dirichlet boundary condition at $x=0$ behaves asymptotically at infinity like the corresponding function of the unperturbed operator, that is,
\be\label{eq:mar}
m(z) = -{\sqrt{-z}}\,(1+\oo(1)),
\ee
as $z\to \infty$ in any nonreal sector in the open upper complex half-plane $\C_+$ (let us stress that the high-energy behavior of $m$ can be deduced from the asymptotic behavior of the corresponding spectral function by using Tauberian theorems for the Stieltjes transform, see, e.g., \cite[Theorem~II.4.3]{ls}). A simple proof of this formula was found by B.\ M.\ Levitan in \cite{lev52} (a short self-contained proof of \eqref{eq:mar} can be found in, e.g., \cite[Lemma~9.19]{tschroe}). An extension of \eqref{eq:mar} to general Sturm--Liouville operators was proposed by I.\ S.\ Kac \cite{ka56, ka71, ka72, ka73} and W.\ N.\ Everitt \cite{ev71}. In particular, it was noticed by I.\ S.\ Kac that the high-energy behavior of the $m$-function of a general Sturm--Liouville operator
\be \label{eqnGenSL}
\frac{1}{w(x)}\left(-\frac{d^2}{dx^2} + q(x)\right)
\ee
depends on the behavior of the weight function $w$ at the left endpoint. Namely, it was shown in \cite{ka73} that 
the Dirichlet $m$-function corresponding to~\eqref{eqnGenSL} satisfies 
\begin{align}\label{eq:kac}
m(z) & = - {d_{\frac{1}{2+\alpha}}}{(-z)^{\frac{1}{2+\alpha}}}(C+\oo(1)), & d_\nu & = \frac{(1-\nu)^{\nu}\Gamma(1-\nu)}{\nu^{1-\nu}\Gamma(\nu)},
\end{align}
as $z\to \infty$ in any nonreal sector in $\C_+$ if 
\be\label{eq:kac2}
\lim_{x\to 0}\,\frac{1}{x^{1+\alpha}}\int_0^x w(t)dt = C^{{2+\alpha}}.
\ee
Here $C>0$ and $\alpha>-1$. The necessity of condition \eqref{eq:kac2} for the validity of \eqref{eq:kac} was proven later by Y.\ Kasahara \cite{ka75} and C.\ Bennewitz \cite{ben} (see also Section \ref{sec:06}). A lot of results on spectral asymptotics for Sturm--Liouville and Dirac operators have been obtained since then due to their numerous applications. For instance, the high-energy behavior of $m$-functions is important in inverse spectral theory \cite{kt2, lev, mar, tschroe} (and the references therein); in the theory of stochastic processes \cite{ka75, ka12, kawa}; in the study of Hardy--Littlewood--P\'olya--Everitt type (HELP) inequalities \cite{Ben87, ben, ev72}; in the similarity problem for indefinite Sturm--Liouville operators \cite{fle14, kkm09, kos13}, etc.  The list of applications as well as the list of references are by no means complete. For further details and historical remarks we refer to the excellent paper by C.\ 
Bennewitz \cite{ben} (see also \cite{ka75, kawa}). 

Our main objective in this article is to prove one-term asymptotic formulas for the $m$-functions of one-dimensional Dirac spectral problems of the form
\be\label{eq:II.01D}
JY' + QY = z H Y
\ee 
on an interval $[0,L)$, where $J=\bigl(\begin{smallmatrix} 0 & -1 \\ 1 & 0 \end{smallmatrix}\bigr)$ and $Q$, $H$ are locally integrable, $2\times2$ real matrix-valued functions on $[0,L)$ so that $Q(x)=Q(x)^*$ and $H(x)=H(x)^*\ge 0$, $H(x)\neq 0$ for almost all $x\in[0,L)$. The analog of the Marchenko formula \eqref{eq:mar} in this case is 
\be\label{eq:mar_D}
m(z) =\I + \oo(1)
\ee
as $z\to \infty$ in any nonreal sector in $\C_+$ (see, e.g., \cite{cg}, \cite[Theorem X.6.4]{ls}). The latter holds whenever $L=\infty$, $H\equiv I_2=\bigl(\begin{smallmatrix} 1 & 0 \\ 0 & 1 \end{smallmatrix}\bigr)$
and $Q\in L^1_{\loc}([0,\infty))$. It was noticed by W.\ N.\ Everitt, D.\ B.\ Hinton and J.\ K.\ Shaw \cite{ehs} that \eqref{eq:mar_D} still remains valid in the sense that
\be\label{eq:ehs}
m(z) =\I \sqrt{\frac{a_0}{c_0}}+ \oo(1)
\ee
whenever 
\be
H = \begin{pmatrix} a & 0 \\ 0 & c \end{pmatrix}, \qquad \lim_{x\to 0}\frac{1}{x}\int_0^x |a(t)-a_0| + |c(t)-c_0| \,dt  =0,
\ee
with some positive constants $a_0$, $c_0>0$. Moreover, the problem was posed to characterize all diagonal matrix-functions $H$ such that the $m$-function satisfies \eqref{eq:ehs}. Some progress was made by H.\ Winkler \cite[Theorem 4.2]{wi} and 
as one of the main results in this article we will solve this problem (see Corollary \ref{cor:4.2} and Theorem \ref{th:3.16}).
Our approach can be seen as a development of the ideas in \cite{ka75} and \cite{ben}. The key role is played by continuity properties of de Branges' correspondence for trace normed canonical systems (see Proposition \ref{prop:dB_corresp}). 

Finally, let us briefly outline the content of the paper. Section \ref{sec:pre} is of preliminary character and it collects basic notions and facts about $2\times 2$ first order canonical systems. Furthermore, in Section \ref{sec:model} we present several prototypical examples when the corresponding $m$-function can be computed explicitly. Our main results are given in Section \ref{sec:03}. Namely we characterize the coefficients of all canonicals systems such that the $m$-functions behave asymptotically at infinity like the $m$-functions considered in Section \ref{sec:model}. In Section \ref{sec:04}, we apply the results of Section \ref{sec:03} to study the high-energy behavior of $m$-functions of radial Dirac operators. This, in particular, allows us to extend the results from \cite{kt2} for radial Schr\"odinger operators to the case of potentials which belong to $H^{-1}_{\loc}([0,\infty))$. In Section \ref{sec:06}, we recover the results of I.\ S.\ Kac \cite{ka73} and Y.\ Kasahara \cite{ka75} on spectral asymptotics for Krein strings. The final Section 
\ref{sec:07} deals with generalized indefinite strings \cite{IndString}, that is, with spectral problems of the form
\be\label{eq:indstr0}
-y'' = z\, \omega\, y + z^2 \dip\, y
\ee
on an interval $[0,L)$ with $L\in(0,\infty]$, a real-valued distribution $\omega \in H^{-1}_{\loc}([0,L))$ and a non-negative Borel measure $\dip$ on $[0,L)$. Again, we prove one-term asymptotic formulas for the corresponding $m$-functions. Let us mention that in the special case when $\dip = 0$ and $\omega$ is a signed measure attracted some attention in the past (see, e.g., \cite{am, ben}). The current interest in spectral problems of the form \eqref{eq:indstr0} is motivated by the study of various nonlinear wave equations for which \eqref{eq:indstr0} serves as an isospectral problem under the corresponding nonlinear flow (see \cite{ConservMP, IndString}). Note that our approach enables us to obtain a complete characterization of coefficients $\omega$ and $\dip$ such that the corresponding $m$-function has a certain asymptotic behavior at infinity (for instance, \eqref{eq:mar_D}). 

The supplementary Appendix \ref{app:osv} collects basic information on the classes of regularly, rapidly and slowly varying functions.

\section{First order canonical systems}\label{sec:cansys}

\subsection{Preliminaries}\label{sec:pre}

In this section, the main object under consideration is the two-dimensional canonical system
\be\label{eq:II.01}
JY' = z HY
\ee
on a finite or infinite interval $[0,L)$, where $z$ is a complex spectral parameter,  
\begin{align}\label{eq:II.02}
J & =\begin{pmatrix} 0 & -1\\ 1 & 0 \end{pmatrix}, &
H(x) & =\begin{pmatrix} a(x) & b(x) \\ b(x) & c(x) \end{pmatrix}, \quad x\in[0,L),  
\end{align}
and the coefficients $a$, $b$ and $c$ of $H$ are real-valued, locally integrable functions on $[0,L)$ so that $H(x)=H(x)^\ast\ge 0$ and $H(x)\not=0$ for almost all $x\in [0,L)$. We assume that the limit-point case prevails at the endpoint $L$, that is,  
\be\label{eq:tr=infty}
 \int_0^L\tr\, H(x)dx=\infty,
\ee 
and furthermore, exclude the case when $b(x)=c(x)=0$ for almost all $x\in[0,L)$. A function $H$ with these properties is called a Hamiltonian on $[0,L)$. 

Associated with the problem~\eqref{eq:II.01} is the Hilbert space $L^2([0,L);H)$, which consists of all (equivalence classes of) vector-valued functions $f:[0,L)\to \C^2$  so that 
\be
\int_0^L f(x)^\ast H(x)f(x) dx <\infty.
\ee
We omit the description of a linear relation in $L^2([0,L);H)$ associated with \eqref{eq:II.01} and its spectral properties. Instead we refer the interested reader to \cite{ka83, lema, ro14}.

Now consider the fundamental matrix solution $\U$ of the canonical system \eqref{eq:II.01},
\begin{align}\label{eq:II.05}
\U(z,x) & = \begin{pmatrix} 
\theta_1(z,x) & \phi_1(z,x) \\
\theta_2(z,x) & \phi_2(z,x)
\end{pmatrix}, & \U(z,0) & = I_2.
\end{align}
The assumptions made imply that 
 for every $z\in\C\backslash\R$ there is a unique (up to a scalar multiple) solution $\Psi(z,\redot)$ to \eqref{eq:II.01} which belongs to $L^2([0,L);H)$. This allows us to introduce the Weyl function ($m$-function or Weyl--Titchmarsh function) $m:\C\backslash\R\to \C$ of the canonical system \eqref{eq:II.01} by requiring that 
\be\label{eq:II.06}
 \U(z,\redot)\begin{pmatrix} 1 \\ m(z) \end{pmatrix}
 \in L^2([0,L); H).
\ee
 The function $m$ is a Herglotz--Nevanlinna function, $m\in (N)$, which means that $m$ is holomorphic in $\C\backslash\R$, maps $\C_+$ into its closure and satisfies $m(z)^* = m(z^*)$. It is known that $m$ thus admits the integral representation
\be\label{eq:II.07}
m(z) = c_1 + c_2z + \int_{\R}\Big(\frac{1}{t-z} - \frac{t}{1+t^2}\Big) d\rho(t),\quad z\in\C\backslash\R,
\ee
where $c_1\in \R$, $c_2\ge 0$ and the function $\rho:\R\to\R$ is non-decreasing and such that
\be\label{eq:II.08}
\int_{\R} \frac{d\rho(t)}{1+t^2}<\infty.
\ee
We normalize $\rho$ by requiring it to be left-continuous and to satisfy $\rho(0)=0$. Under this normalization, $\rho$ is unique. It is called the spectral function and the corresponding measure $d\rho$ is called the spectral measure of the canonical system \eqref{eq:II.01}. 

\begin{remark}\label{rem:II.01}
 In view of condition \eqref{eq:tr=infty}, an equivalent definition of the $m$-function is given by
\be\label{eq:II.10}
m(z)= - \lim_{x\to L} \frac{\theta_1(z,x)\xi(z) + \theta_2(z,x)}{\phi_1(z,x)\xi(z) + \phi_2(z,x)},\quad z\in \C\backslash\R.
\ee
Here $\xi$ is an arbitrary fixed Herglotz--Nevanlinna function and the limit in \eqref{eq:II.10} does not depend on $\xi$.
\end{remark}

A description of all possible $m$-functions of canonical systems \eqref{eq:II.01} was obtained by L.\ de Branges \cite{dBii} (see also \cite{wi95}). In order to state this result, we say that the Hamiltonian  $H$ is trace normed if 
\be\label{eq:II.03}
 \tr\, H(x) = a(x) + c(x)=1 
\ee
for almost all $x\in[0,L)$. 

\begin{theorem}[de Branges]\label{th:dB01}
Each Herglotz--Nevanlinna function is the $m$-function of a canonical system \eqref{eq:II.01} with a trace normed Hamiltonian $H$ on $[0,\infty)$. Upon identifying Hamiltonians that coincide almost everywhere, this correspondence is one-to-one.   
\end{theorem}

The next result contains a continuity property for de Branges' correspondence in Theorem \ref{th:dB01}. 

\begin{proposition}[\cite{dBii}]\label{prop:dB_corresp}
Let $H$, $H_n$ be trace normed Hamiltonians on $[0,\infty)$ and $m$, $m_n$ be the corresponding $m$-functions for every $n\in\N$. Then the following conditions are equivalent:
\begin{enumerate}[label=\emph{(\roman*)}, ref=(\roman*), leftmargin=*, widest=iiii]
\item \label{itdBcor1} We have  
\be\label{eq:conv}
\int_0^x H_n(t)dt \to \int_0^x H(t)dt
\ee
locally uniformly for all $x\in[0,\infty)$,
\item For every $x\in[0,\infty)$ we have $\U_n(\ledot,x)\to \U(\ledot,x)$ locally uniformly on $\C$.
\item \label{itdBcor3} We have $m_n\to m$ locally uniformly on $\C\backslash\R$.
\end{enumerate}
\end{proposition}

\begin{remark}\label{rem:dBth}
Note that the locally uniform convergence in \eqref{eq:conv} can be replaced by simple pointwise convergence upon employing a compactness argument. Clearly, we only need to show that pointwise convergence in \eqref{eq:conv} implies locally uniform convergence. Since the functions $\int_0^x H_n(t)dt$ are equicontinuous and  uniformly bounded on each compact interval, the Arzel\'{a}--Ascoli compactness criterion implies that this set of functions is pre-compact in $C[0,c]$ for all $c>0$. However, pointwise convergence in \eqref{eq:conv} implies that this sequence has only one possible limit point and hence the convergence in \eqref{eq:conv} is locally uniform.   
\end{remark}

\begin{remark}
Let us mention that \ref{itdBcor1} also implies \ref{itdBcor3} if one drops the trace normalization assumption.
\end{remark}

To complete this subsection, let us consider the spectral problem \eqref{eq:II.01} also without the trace normalization assumption \eqref{eq:II.03}. The next result is an analog of Theorem \ref{th:dB01} (see, e.g., \cite{ekt,wowi}) in this case.

\begin{lemma}\label{th:uniq}
Let $H_1$, $H_2$ be two Hamiltonians on $[0,L_1)$, $[0,L_2)$, respectively, and $m_1$, $m_2$ be the corresponding $m$-functions. Then we have $m_1=m_2$ if and only if there is a locally absolutely continuous bijection $\eta:[0,L_1)\to [0,L_2)$ such that 
\be\label{eq:II.eta}
H_1(x)=\eta'(x)H_2(\eta(x))
\ee
for almost all $x\in[0,L_1)$. 
In particular, if $H_2$ is trace normed, then $\eta' = \tr\, H_1$.
\end{lemma}

Finally, we will need the following simple fact.

\begin{lemma}\label{lem:scaling}
Let $H$ be a Hamiltonian on $[0,L)$ and $m$ be the corresponding $m$-function. If $r_1$, $r_2$, $r_3$ are fixed positive numbers and the Hamiltonian $H_r$ on $[0,L/r_1)$ is defined by 
\be
 H_r(x)=r_2\begin{pmatrix}
r_3a(r_1x) & b(r_1x)\\
b(r_1x) & \frac{1}{r_3}c(r_1x)
\end{pmatrix}
\ee
for almost all $x\in[0,L/r_1)$, then 
the corresponding $m$-function is given by 
\be
 m_r(z)=r_3\, m\Big(\frac{r_2}{r_1} z\Big),\quad z\in\C\backslash\R.
\ee
\end{lemma}

\begin{proof}
It suffices to notice that
\[
 \U_r(z,x)=\begin{pmatrix}
\theta_1(\frac{r_2}{r_1} z,r_1x) & \frac{1}{r_3}\phi_1(\frac{r_2}{r_1} z,r_1x)\\
r_3\theta_2(\frac{r_2}{r_1} z,r_1x) & \phi_2(\frac{r_2}{r_1} z,r_1x)
\end{pmatrix}, \quad x\in[0,L/r_1),~z\in\C, 
\]
is the fundamental matrix solution for \eqref{eq:II.01} with $H_r$ in place of $H$.
\end{proof}

Let us mention that there are a lot of results which connect the asymptotic behavior of the Weyl function $m$ and the corresponding spectral function $\rho$. We need the following result (see Theorem 7.5 in \cite{ben}).

\begin{theorem}[\cite{ben}]\label{th:tauber}
Let $m$ and $\ti{m}$ be Herglotz--Nevanlinna functions with the spectral functions $\rho$ and $\ti{\rho}$, respectively, and let $f$ be a positive function defined for large $r$ so that $f(r)/r$ decreases as $r\to \infty$. If 
\be\label{eq:tauberm}
m(r\mu) = f(r)\ti{m}(\mu)(1+\oo(1)),\quad r\to \infty,
\ee
holds locally uniformly for non-real $\mu$, then
\be\label{eq:tauberr}
\rho(\pm t) = tf(t) (\tilde{\rho}(\pm 1) + \oo(1)), \quad t\to +\infty,
\ee
except if $\ti{\rho}$ has a jump at $0$. In this case, $\rho$ satisfies
\be
\rho(t) - \rho(-t) = tf(t) (\tilde{\rho}(1) - \tilde{\rho}(-1) +\oo(1)),\quad t\to +\infty.
\ee 
 \end{theorem} 
 
\begin{remark} \label{rem:2.9}
The converse of Theorem~\ref{th:tauber} is not true in general as the example $m(z)= \alpha_- \log(z) - \alpha_+ \log(-z)$ shows.
The corresponding measure is $\rho(t) = \alpha_\pm t$,  $\pm t\geq 0$, and \eqref{eq:tauberr} holds (choosing $m=\tilde{m}$) with $f(t)=1$ but \eqref{eq:tauberm} does not hold. However, under  additional assumptions the converse is true as well, as, for example, various Tauberian theorems for the Stieltjes transform show (see \cite{bgt, Kor04}). 
\end{remark}

\subsection{Examples}\label{sec:model}
In this subsection we present three model examples which play an important role in our future considerations.

\begin{example}
Consider the system \eqref{eq:II.01} with the Hamiltonian 
\be\label{eq:H_0}
H(x) = H_0:=\begin{pmatrix} a_0 & b_0 \\ b_0 & c_0 \end{pmatrix}, \quad x\in [0,\infty), 
\ee
where $a_0\geq0$, $b_0\in \R$ and $c_0>0$ are constants so that 
\be
h_0^2:=\det H_0 = a_0c_0 - b_0^2\geq0.
\ee
The fundamental matrix solution $\U_0$ in this case is simply 
\be\label{eq:u_0}
\U_0(z,x) = \E^{-zJH_0x} = \cos(zh_0x)I_2 - \frac{\sin(zh_0x)}{h_0} JH_0, \quad x\in[0,\infty),~z\in\C,  
\ee
where the fraction on the right-hand side has to be read as $zx$ if $h_0=0$. 
From this one readily infers that on $\C_+$, the $m$-function $m_0$ is equal to a constant $\zeta_0$ given by 
\be\label{eq:m_1}
\zeta_0 = \frac{\I h_0 - b_0}{c_0} = \frac{-a_0}{b_0+\I h_0}. 
\ee
Consequently, the corresponding spectral function is given by
\be\label{eq:rho_1}
\rho_0(t) = \frac{1}{\pi} \frac{h_0}{c_0}t,\quad t\in\R. 
\ee

Note that for each constant $\zeta_0$ in the closure of $\C_+$ there is a unique trace normed matrix $H_0$ so that the corresponding $m$-function is equal to $\zeta_0$ on $\C_+$. 
In fact, the entries of this matrix are explicitly given by  
\begin{align}\label{eqn31rel}
  a_0 & = \frac{|\zeta_0|^2}{|\zeta_0|^2+1}, & b_0 & = \frac{-\re\,\zeta_0}{|\zeta_0|^2+1}, & c_0 & = \frac{1}{|\zeta_0|^2+1}.  
\end{align}
Of course this fact is reminiscent of Theorem~\ref{th:dB01}. 
\end{example}

\begin{example}
Consider the system \eqref{eq:II.01} with the Hamiltonian 
\be\label{eq:R_a1}
H(x) = H_\alpha(x):=\begin{pmatrix}
p_{\alpha}(x) & 0\\
0 & p_{-\alpha}(x)
\end{pmatrix}, \quad x\in[0,\infty), 
\ee
where the function $p_\alpha:[0,\infty)\to[0,\infty)$ is defined by
\begin{align}\label{eq:R_a2}
 p_{\alpha}(x) & =\begin{cases}
(1+\alpha)x^{\alpha}, & \alpha\ge 0,\\
1, & \alpha<0,
\end{cases} \quad x\in[0,\infty). 
\end{align}
Clearly it suffices to look at the case $\alpha\ge 0$ since the case  $\alpha< 0$ follows by
exchanging rows and columns in the fundamental matrix solution. In this case the solutions of \eqref{eq:II.01} are connected with solutions of the scalar ordinary differential equation (see, e.g., \cite[Lemma 4.6]{ben}, \cite{ez,ka75})
\begin{align}\label{eq:sp_ex2}
 -y'' = \zeta\, p_{\alpha}y, \qquad \zeta=z^2,
\end{align}
which has a system of entire solutions given by (cf.\ \cite[(10.13.3)]{dlmf})
\begin{align}\begin{split}
\widetilde{\theta}_\alpha(\zeta,x)= 
\sqrt{x}\, \frac{\Gamma(1-\nu)}{(\nu-\nu^2)^{\frac{\nu}{2}}} \zeta^{-\frac{\nu}{2}}&J_{-\nu}\big(2\sqrt{(\nu-\nu^2)\zeta}\, x^{\frac{1}{2\nu}}\big)\\
&={}_0F_1\big(1-\nu, (\nu^2-\nu) \zeta x^{\frac{1}{\nu}}\big), \end{split} \\ \begin{split}
\widetilde{\phi}_\alpha(\zeta,x)=
\sqrt{x}\, \Gamma(1+\nu) \big(\sqrt{(\nu-\nu^2) \zeta}\big)^{\nu}&J_{\nu}\big(2\sqrt{(\nu-\nu^2)\zeta}\, x^{\frac{1}{2\nu}}\big)\\
&= x\, {}_0F_1(1+\nu, (\nu^2-\nu) \zeta x^{\frac{1}{\nu}}),
\end{split}\end{align}
where $1/\nu=2+\alpha$, $\Gamma$ is the usual Gamma function, $J_\nu$ is the Bessel function of the first kind and 
\begin{align}
{}_0F_1(c,z)=\sum_{k=0}^\infty \frac{\Gamma(c)}{\Gamma(k+c)} \frac{z^k}{k!}
\end{align}
 is the hypergeometric function.
Using \cite[(10.2.5), (10.4.7)]{dlmf}, we conclude that the $m$-function $\ti{m}_\alpha$ defined by 
\begin{align}
\widetilde{\psi}_\alpha(\zeta,\redot) = \widetilde{\theta}_\alpha(\zeta,\redot) +\ti{m}_\alpha(\zeta) \widetilde{\phi}_\alpha(\zeta,\redot) \in L^2(\R_+; p_\alpha)
\end{align}
is given by
\begin{align}\label{eq:m_alpha2C}
\widetilde{m}_\alpha(\zeta) & = -d_{\frac{1}{2+\alpha}} (-\zeta)^{\frac{1}{2+\alpha}}, &
d_\nu & = \frac{(1-\nu)^{\nu}\Gamma(1-\nu)}{\nu^{1-\nu}\Gamma(\nu)}.
\end{align}
Thus the fundamental matrix solution of \eqref{eq:II.01} with $H=H_{\alpha}$ for $\alpha\ge 0$ reads
\be
U(z,x) = \begin{pmatrix}
\widetilde{\theta}_\alpha(z^2,x) & z\widetilde{\phi}_\alpha(z^2,x)\\
z^{-1}\widetilde{\theta}_\alpha'(z^2,x) & \widetilde{\phi}_\alpha'(z^2,x)
\end{pmatrix},\quad x\in[0,\infty),~z\in \C,
\ee
and hence the corresponding $m$-function $m_\alpha$ is given by $m_\alpha(z)=\frac{1}{z} \widetilde{m}_\alpha(z^2)$.
Using $m_{-\alpha}(z) = -1/m_{\alpha}(z)=1/m_{\alpha}(-z)$ we finally obtain
\be\label{eq:m_alpha2}
m_\alpha(z)=\begin{cases}
-d_{\frac{1}{2+\alpha}} \E^{-\I \pi\frac{1}{2+\alpha}}z^{-\frac{\alpha}{2+\alpha}}, & \alpha>0,\\
\quad \I, & \alpha=0,\\
d_{\frac{1+|\alpha|}{2+|\alpha|}} \E^{\I \pi\frac{1}{2+|\alpha|}}z^{\frac{|\alpha|}{2+|\alpha|}}, & \alpha<0,
\end{cases}
\quad z\in \C_+.
\ee 
Here the branch cut of the root is taken along the negative semi-axis as usual.
Finally, the corresponding spectral function $\rho_\alpha$ is given by
\be\label{eq:rho_alpha}
\rho_\alpha(t) = \rho_\alpha(-t) = \frac{1}{\pi}\begin{cases}
\frac{2+\alpha}{2}\sin(\frac{\pi}{2+\alpha})d_{\frac{1}{2+\alpha}}\, t^{\frac{2}{2+\alpha}}, & \alpha>0,\\
t, & \alpha=0,\\
\frac{2+|\alpha|}{2+2|\alpha|}\sin(\frac{\pi}{2+|\alpha|})d_{\frac{1+|\alpha|}{2+|\alpha|}}\, t^{\frac{2+2|\alpha|}{2+|\alpha|}}, & \alpha<0,
\end{cases}\quad t\ge 0.
\ee

Note that the trace normed Hamiltonian $\tilde{H}_\alpha$ corresponding to the $m$-function $m_{\alpha}$ is given by  (cf.\ Lemma \ref{th:uniq}) 
\begin{align}\label{eq:R_a3}
 \ti{H}_\alpha(x) = \frac{1}{\eta_\alpha'(\xi_\alpha(x))}\begin{pmatrix} p_{\alpha}(\xi_\alpha(x)) & 0 
 \\ 0 & p_{-\alpha}(\xi_\alpha(x)) \end{pmatrix},\quad x\in [0,\infty),
\end{align} 
where $\eta_\alpha'=\tr\,H_\alpha$ and $\xi_\alpha$ is the inverse of $\eta_\alpha$, that is,
\begin{align}\label{eq:R_a4}
\eta_\alpha'(x) & =1+(1+|\alpha|)x^{|\alpha|}, & \xi_{\alpha}(x)+(\xi_\alpha(x))^{1+|\alpha|} & =x.
\end{align}
 For this Hamiltonian we have 
 \begin{align}\label{eq:R_a5}
 \ti{H}_0(x) & =\frac{1}{2}H_0(x), & \ti{H}_\alpha(x) & = H_\alpha(x)+\oo(H_\alpha(x))\ \ \text{as}\ \ x\to 0,\quad \alpha\neq 0. 
 \end{align}
\end{example}

\begin{example}
Consider the system \eqref{eq:II.01} with the Hamiltonian  
\be\label{eq:ex2}
H(x)=\begin{pmatrix} \id_{[1,\infty)}(x) & 0 \\ 0 & \id_{[0,1)}(x) \end{pmatrix},\quad x\in[0,\infty).
\ee
It is immediate to check that the fundamental matrix solution is
\be
\U(z,x)=\begin{pmatrix} 
1 & z - z(1-x)\id_{[0,1)}(x) \\ z(1-x)\id_{[1,\infty)}(x) & 1+ z^2(1-x)\id_{[1,\infty)}(x)
\end{pmatrix},\quad x\in[0,\infty),~z\in\C.
\ee
Therefore, the corresponding $m$-function is given by 
\be\label{eq:m_ex2}
m(z)=-\frac{1}{z},\quad z\in\C\backslash\R,
\ee
and the corresponding spectral function is a step function with a jump at zero
\be\label{eq:rho_ex2}
\rho(t) = \id_{(0,\infty)}(t),\quad t\in\R.
\ee

Also note that the $m$-function corresponding to the Hamiltonian 
\be
\tilde{H}(x)=\begin{pmatrix} \id_{[0,1)}(x) & 0 \\ 0 & \id_{[1,\infty)}(x) \end{pmatrix}, \quad x\in [0,\infty),
\ee
is simply given by $\tilde{m}(z)=-1/m(z) = z$.
\end{example}

\section{Spectral asymptotics for canonical systems}\label{sec:03}

The main aim of the current section is to characterize Hamiltonians $H$ for which the corresponding Weyl function $m$ of \eqref{eq:II.01} behaves asymptotically at infinity like the function $m_\alpha$ given by \eqref{eq:m_alpha2}. 

\subsection{The case \texorpdfstring{$\alpha=0$}{alpha=0}} \label{sec:5.1}

To explain our approach, we start with the simplest case when the $m$-function behaves asymptotically at infinity like a constant.

\begin{theorem}\label{th:4.1}
Let $H$ be a trace normed Hamiltonian on $[0,\infty)$ and $m$ be the corresponding $m$-function. Then the following conditions are equivalent:
\begin{enumerate}[label=\emph{(\roman*)}, ref=(\roman*), leftmargin=*, widest=iii]
\item \label{it3.1i} For some $a_0\geq0$, $b_0\in\R$ and $c_0>0$ with $h_0^2=a_0c_0-b_0^2\geq 0$ one has
\be\label{eq:5.01}
 \frac{1}{x} \int_0^x H(t) dt \rightarrow \begin{pmatrix} a_0 & b_0 \\ b_0 & c_0 \end{pmatrix}, \quad x\to 0. 
\ee
\item \label{it3.1ii} For some $\zeta_0$ in the closure of $\C_+$ one has 
\be\label{eq:5.02}
m(z) \rightarrow \zeta_0
\ee
 as $z\rightarrow\infty$ in any closed sector in $\C_+$.
\end{enumerate}
In this case, the quantities $a_0$, $b_0$, $c_0$ and $\zeta_0$ are related as in~\eqref{eq:m_1} and~\eqref{eqn31rel}.  
\end{theorem}

\begin{proof}
\ref{it3.1i} $\Rightarrow$ \ref{it3.1ii}. Upon setting 
\begin{align*}
 H_0 & = \begin{pmatrix} a_0 & b_0 \\ b_0 & c_0 \end{pmatrix}, & H_r(x) & = H\Big(\frac{x}{r}\Big),\quad x\in[0,\infty),
\end{align*}
for all $r>0$, we obtain from \eqref{eq:5.01} that  
\[
\int_0^x H_r(t)-H_0\,dt = \int_0^x H(t/r)-H_0\,dt= r\int_0^{x/r} H(t)dt - H_0 x\rightarrow 0
\]
as $r\to\infty$ for every $x\in[0,\infty)$. By Proposition \ref{prop:dB_corresp}, the corresponding $m$-functions $m_r$ converge locally uniformly to the function $m_0$ given by \eqref{eq:m_1}. 
 However, in view of Lemma \ref{lem:scaling} we have 
\[
m_r(\mu)=m(r\mu),\quad \mu \in\C_+,\ r>0,
\]
and hence  
\[
m(r\mu)=m_r(\mu) \rightarrow m_0(\mu) = \frac{\I h_0 -b_0}{c_0}, \quad r\to\infty, 
\]
where the convergence holds uniformly for all $\mu$ in any compact subset in $\C_+$.

\ref{it3.1ii} $\Rightarrow$ \ref{it3.1i}. Let $a_0$, $b_0$, $c_0$ be defined by~\eqref{eqn31rel} such that the relation~\eqref{eq:m_1} holds. 
With the notation from the first part of the proof, we infer from~\eqref{eq:5.02} that the $m$-functions $m_r$ converge locally uniformly to the function $m_0$ as $r\to\infty$. 
Consequently, we have 
\[
 \frac{1}{x} \int_0^{x} H(t)dt - H_0 = \int_0^1 H(xt) - H_0\,dt = \int_0^1 H_{\frac{1}{x}}(t) - H_0\,dt \to0
\]
as $x\to0$ in view of Proposition~\ref{prop:dB_corresp}. 
\end{proof}
 
\begin{corollary}\label{cor:4.2}
Let $H$ be a Hamiltonian on $[0,L)$ and $m$ be the corresponding $m$-function. Then the following conditions are equivalent:
\begin{enumerate}[label=\emph{(\roman*)}, ref=(\roman*), leftmargin=*, widest=iii]
\item  For some $a_0\geq0$, $b_0\in\R$ and $c_0>0$ with $h_0^2=a_0c_0-b_0^2\geq 0$ one has
\begin{align}\label{eq:5.21}
  \frac{1}{\eta(x)} \int_0^{x} H(t)dt &  \rightarrow \begin{pmatrix}
a_0 & b_0\\
b_0 & c_0
\end{pmatrix}, \quad x\to 0, & \eta(x) & =\int_0^x\tr\, H(t)dt.
\end{align}
\item
For some $\zeta_0$ in the closure of $\C_+$ one has~\eqref{eq:5.02} 
as $z\rightarrow\infty$ in any closed sector in $\C_+$.
\end{enumerate}
In this case, the quantities $a_0$, $b_0$, $c_0$ and $\zeta_0$ are related as in~\eqref{eq:m_1} and~\eqref{eqn31rel}. 
\end{corollary}

\begin{proof}
Consider the trace normed Hamiltonian $\ti{H}$ on $[0,\infty)$ defined by 
\[
\ti{H}(x) = \eta'(\xi(x))^{-1} H(\xi(x)),\quad x\in[0,\infty),
\]
where $\xi$ is the inverse of $\eta$. By Lemma \ref{th:uniq} the $m$-functions of $H$ and $\ti{H}$ coincide and it remains to apply Theorem \ref{th:4.1}.
\end{proof}

\begin{corollary}\label{cor:4.3}
Let $H$ be a Hamiltonian on $[0,L)$ and $\rho$ be the corresponding spectral function. 
If for some $a_0\geq0$, $b_0\in\R$ and $c_0>0$ with $h_0^2=a_0c_0-b_0^2\geq 0$ one has~\eqref{eq:5.21}, then
\be
\rho(t) =  \frac{1}{\pi}\frac{h_0}{c_0} t +\oo(t),\quad t\to\pm \infty. 
\ee
\end{corollary}

\begin{proof}
The proof follows by combining Corollary \ref{cor:4.2} and formula \eqref{eq:rho_1} with the Tauberian Theorem~\ref{th:tauber}. 
\end{proof}

Clearly, the converse of Corollary \ref{cor:4.3} is not true since different Hamiltonians $H_0$ given by \eqref{eq:H_0} have the same spectral function \eqref{eq:rho_1} (see also Remark \ref{rem:2.9}).

\subsection{The case \texorpdfstring{$\alpha\neq 0$}{alpha<>0}}\label{sec:5.2}
 
Given $H$ as in \eqref{eq:II.02} we set
\begin{align}\label{eq:ABC}
A(x) & =\int_0^x a(t)dt, &  B(x) & =\int_0^x b(t)dt, &  C(x) & =\int_0^x c(t)dt.
\end{align}
We begin with the case $\alpha>0$.

\begin{definition}\label{def:f_func}
Let $F:[0,\infty)\to [0,\infty)$ be the inverse of $x\mapsto 1/(xA(x))$ and set
\be\label{eq:f_func}
f(r)=rF(r^2),\quad r>0.
\ee
\end{definition}

By definition, we have $A(x) = 1/(xF^{-1}(x))$ and 
\be\label{eq:f=F}
rf(r)A\Big(\frac{f(r)}{r}\Big)=r^2F(r^2)A(F(r^2))=1,\quad r>0.
\ee
It is easy to see that $F(r)\downarrow  0$ as $r\uparrow \infty$. Moreover, $f(r)\uparrow \infty$ as $r\uparrow \infty$ whenever $A(x)=\oo(x)$ as $x\downarrow 0$.

To formulate our theorem we recall the set of regularly varying functions $RV_\alpha$ defined in Definition~\ref{def:f_rv}.

\begin{theorem}\label{th:}
Let $H$ be a  trace normed Hamiltonian on $[0,\infty)$ and $m$ be the corresponding $m$-function \eqref{eq:II.06}.
Then, for given $\alpha>0$, the following conditions are equivalent:
\begin{enumerate}[label=\emph{(\roman*)}, ref=(\roman*), leftmargin=*, widest=iii]
\item\label{italphai} The coefficients \eqref{eq:ABC} satisfy
\be\label{eq:5.07c}
A\in RV_{1+\alpha}(0)\quad\text{and}\quad B(x)=\oo(\sqrt{xA(x)}),\quad x\to 0.
\ee
\item\label{italphaii} For some strictly increasing function $f\in RV_{\frac{\alpha}{2+\alpha}}(\infty)$ and $m_\alpha$ from \eqref{eq:m_alpha2} 
\be\label{eq:5.07}
m(r\mu) = \frac{m_\alpha(\mu)}{f(r)}(1+\oo(1)),\quad r\to \infty,
\ee
where the estimate holds uniformly for $\mu$ in any compact subset of $\C$ not intersecting the real line.
\end{enumerate}
In this case the function $f$ and the Hamiltonian $H$ are connected via Definition \ref{def:f_func}. 
\end{theorem}

\begin{proof}
\ref{italphai} $\Rightarrow$ \ref{italphaii}. Assume that $A\in RV_{1+\alpha}(0)$ and let $f$ and $F$ be the functions in Definition \ref{def:f_func}. Clearly, we have 
\[
F\in RV_{\frac{-1}{2+\alpha}}(\infty)\quad \text{and} \quad f\in RV_{\frac{\alpha}{2+\alpha}}(\infty). 
\]
Using Lemma \ref{lem:scaling} with 
\begin{align*}
r_1 & :=\frac{f(r)}{r}=F(r^2), & r_2 & :=f(r), & r_3 & :=f(r),
\end{align*}
 we obtain a family of scaled Hamiltonians $H_r$ with corresponding $m$-functions $m_r$
\begin{align}\label{eq:5.10}
H_r(x) & =\begin{pmatrix}
f(r)^2 a(f(r)x/r) & f(r)b(f(r)x/r)\\
f(r)b(f(r)x/r) & c(f(r)x/r)
\end{pmatrix}, & m_r(\mu) & = f(r) m(\mu r).
\end{align}
Note that $H_r$ is no longer trace normed. Since $A\in RV_{1+\alpha}(0)$, we get 
\be\label{eq:5.37}
\int_0^x H_r(t)dt \to \int_0^x H_\alpha(t)dt=\begin{pmatrix}
x^{1+\alpha} & 0\\
0 & x
\end{pmatrix}
\ee
for every $x>0$ as $r\to \infty$. Indeed, using \eqref{eq:f=F}, we compute
\begin{align*}
\int_0^x f(r)^2 a(f(r)t/r)dt = rf(r)A\big(xF(r^2)\big)= \frac{A\big(xF(r^2)\big)}{A\big(F(r^2)\big)}= x^{1+\alpha}(1+\oo(1))
\end{align*}
as $r\to \infty$. Furthermore, since $\tr\, H(x)=1$, we get
\begin{align*}
\int_0^x c(f(r)t/r)dt = x-\int_0^x a(f(r)t/r)dt = x - \frac{A(xF(r^2))}{F(r^2)} = x(1+\oo(1))
\end{align*}
as $r\to \infty$. Here we also used the fact that $A(x) = \oo(x)$ as $x\to 0$ since $A\in RV_{1+\alpha}(0)$ with $\alpha>0$. Finally, in view of the second condition in \eqref{eq:5.07c} and employing \eqref{eq:f=F} once again, we obtain 
\begin{align*}
\int_0^x f(r)b(f(r)t/r)dt =r B(xF(r^2))=\oo\Big(r\sqrt{xF(r^2) A\big(xF(r^2)\big)}\Big)=x^{1+\frac{\alpha}{2}}\oo(1)
\end{align*}
as $r\to\infty$. Therefore, by Proposition \ref{prop:dB_corresp} and Remark \ref{rem:dBth}, we have 
\be
m_r(\mu)=f(r)m(r\mu)= m_\alpha(\mu)(1+\oo(1)), \quad r\to \infty.
\ee

\ref{italphaii} $\Rightarrow$ \ref{italphai}. Now assume that $m$ satisfies \eqref{eq:5.07} for some strictly increasing function $f\in RV_{\frac{\alpha}{2+\alpha}}(\infty)$. Consider the scaled Hamiltonians $H_r$ defined by \eqref{eq:5.10} with $F(r):=f(\sqrt{r})/\sqrt{r}$, $r>0$. In view of \eqref{eq:5.07}, the corresponding $m$-functions $m_r$ converge uniformly on compact subsets to $m$ as $r\to \infty$. By Proposition \ref{prop:dB_corresp} and Lemma \ref{th:uniq}, we conclude that
\be
\int_0^{\xi_r(x)} H_r(t)dt \to \int_0^{\xi_\alpha(x)} H_\alpha(t)dt,\quad r\to\infty,
\ee
locally uniformly on $[0,\infty)$. Here $\xi_r(x)$ and $\xi_\alpha(x)$ are, respectively, the inverse functions to 
\[
\eta_r(x) = \int_0^x \tr\, H_r(t)dt,\quad \eta_\alpha(x) = \int_0^x \tr\, H_\alpha(t)dt=x+x^{1+\alpha},\quad x>0.
\]
 The latter implies that 
\begin{align}\label{eq:5.41}
\begin{split}
\lim_{r\to\infty}rf(r)A(F(r^2)\xi_r(x))= \xi_{\alpha}(x)^{1+\alpha},\quad 
\lim_{r\to\infty} \frac{C(F(r^2)\xi_r(x))}{F(r^2)} = \xi_{\alpha}(x)
  \end{split}
  \end{align}
  and 
 \begin{align}\label{eq:5.41c}
 \lim_{r\to\infty}  rB(F(r^2)\xi_r(x))= 0 
\end{align}
 locally uniformly on $[0,\infty)$. Noting that $A(x)+C(x)=x$ and using the first relation in \eqref{eq:5.41}, we get
\begin{align*}
\xi_r(x) - \frac{C(F(r^2)\xi_r(x))}{F(r^2)}  = & \frac{A(F(r^2)\xi_r(x))}{F(r^2)} = \frac{rf(r)A(F(r^2)\xi_r(x))}{f(r)^2} \\
&= \frac{\xi_{\alpha}(x)^{1+\alpha}(1+\oo(1))}{f(r)^2}=\xi_{\alpha}(x)^{1+\alpha}\oo(1)
\end{align*}
as $r\to \infty$. In view of the second relation in \eqref{eq:5.41}, we conclude that 
\be\label{eq:5.42}
\xi_r(x) =\xi_\alpha(x)(1+\oo(1))
\ee
locally uniformly as $r\to\infty$. Since $\xi_\alpha(x)$ is the inverse to $\eta_\alpha(x)=x+x^{1+\alpha}$ and $\eta_\alpha'(x) \ge 1$ on $[0,\infty)$, we conclude that $\eta_r(x)\to \eta_\alpha(x)$ locally uniformly on $[0,\infty)$ as $r\to \infty$. The latter implies that
\[
r{f(r)}A(F(r^2)x) \to x^{\alpha+1},\quad r\to \infty,
\]
locally uniformly on $[0,\infty)$. Setting $x=1$,
we get
\[
A(F(r^2)) = \frac{1}{rf(r)}(1+\oo(1)) = \frac{1}{r^2F(r^2)}(1+\oo(1))
\]
as $r\to \infty$. Hence 
\be\label{eq:5.43}
A(x) = \frac{1}{xF^{-1}(x)}(1+\oo(1)),\quad x\to 0.
\ee
Upon observing that $F^{-1}\in RV_{-(2+\alpha)}(0)$ and applying Lemma \ref{lem:a.asymp}, we conclude that $A\in RV_{1+\alpha}(0)$.

Finally, noting that \eqref{eq:5.41c} and \eqref{eq:5.42} hold locally uniformly on $[0,\infty)$, one can show that
$rB(F(r^2))\to 0$ as $  r\to\infty$. 
In view of \eqref{eq:5.43}, we conclude that $B$ satisfies the second condition in \eqref{eq:5.07c}.
 \end{proof}

\begin{corollary}\label{cor:3.08}
 With the assumptions of Theorem \ref{th:}, for given $\alpha>0$, the following conditions are equivalent:
\begin{enumerate}[label=\emph{(\roman*)}, ref=(\roman*), leftmargin=*, widest=iii]
\item The coefficients \eqref{eq:ABC} satisfy
\be\label{eq:5.07d2}
A(x) =x^{1+\alpha}(d+\oo(1))\quad\text{and}\quad B(x)=\oo(x^{\frac{2+\alpha}{2}}),\ \ x\to 0.
\ee
\item With $m_\alpha$ from \eqref{eq:m_alpha2}, the $m$-function \eqref{eq:II.06} satisfies
\be\label{eq:5.07e1}
m(z)=  m_\alpha(z)(d^{\frac{1}{2+\alpha}}+\oo(1)),
\ee
as $z\to \infty$ in any closed sector in $\C_+$. 
\end{enumerate}
\end{corollary}

\begin{proof}
The proof immediately follows from Theorem \ref{th:}. We only need to mention that under the above assumptions
\[
f(r)=d^{-\frac{1}{2+\alpha}}r^{\frac{\alpha}{2+\alpha}}(1+\oo(1)),\quad r\to\infty.\qedhere
\]
\end{proof}

Using Theorem \ref{th:} we can easily investigate the case $\alpha<0$ as well. Indeed, it suffices to note that the change of the Hamiltonian $\tilde{H}=-JHJ$ implies that the corresponding $m$-functions are connected by 
\begin{align}
\tilde{m}(z)={1}/{m(-z)},\quad z\in\C_+.
\end{align}

\begin{definition}\label{def:g_func}
Let $G:[0,\infty)\to [0,\infty)$ be the inverse of $x\mapsto 1/(xC(x))$ and set
\be\label{eq:g_func}
g(r)=rG(r^2),\quad r>0.
\ee
\end{definition}


\begin{theorem}\label{cor:}
 With the assumptions of Theorem \ref{th:}, for given $\alpha<0$, the following conditions are equivalent:
\begin{enumerate}[label=\emph{(\roman*)}, ref=(\roman*), leftmargin=*, widest=iii]
\item The coefficients \eqref{eq:ABC} satisfy
\be\label{eq:5.07d3}
C\in RV_{1+|\alpha|}(0)\quad\text{and}\quad B(x)=\oo(\sqrt{xC(x)}),\ \ x\to 0.
\ee
\item For some strictly increasing function $g\in RV_{\frac{\alpha}{2+\alpha}}(\infty)$ and $m_\alpha$ from \eqref{eq:m_alpha2}, the $m$-function \eqref{eq:II.06} satisfies
\be\label{eq:5.07e2}
m(r\mu)= m_\alpha(\mu)g(r)(1+\oo(1)),\quad r\to \infty,
\ee
where the estimate holds uniformly for $\mu$ in any compact subset of $\C$ not intersecting the real line.
\end{enumerate}
In this case the function $g$ and the Hamiltonian $H$ are connected via Definition \ref{def:g_func}. 
\end{theorem}

\begin{corollary}\label{cor:3.08B}
 With the assumptions of Theorem \ref{th:}, for given $\alpha>0$, the following conditions are equivalent:
\begin{enumerate}[label=\emph{(\roman*)}, ref=(\roman*), leftmargin=*, widest=iii]
\item The coefficients \eqref{eq:ABC} satisfy
\be\label{eq:5.07d}
C(x)  = x^{1+\alpha}(d+\oo(1))\quad\text{and}\quad B(x)=\oo(x^{\frac{2+\alpha}{2}}),\ \ x\to 0,
\ee
\item With $m_\alpha$ from \eqref{eq:m_alpha2}, the $m$-function \eqref{eq:II.06} satisfies
\be\label{eq:5.07e3}
m(z)=  m_{-\alpha}(z)(d^{-\frac{1}{2+\alpha}}+\oo(1)), 
\ee
as $z\to \infty$ in any closed sector in $\C_+$. 
\end{enumerate}
\end{corollary}

We complete this subsection by presenting the high-energy asymptotic for spectral functions.

\begin{corollary}\label{cor:rho_a}
Let $H$ be a trace normed Hamiltonian on $[0,\infty)$ and $\rho$ be the corresponding spectral function. If the coefficients \eqref{eq:ABC} satisfy \eqref{eq:5.07c} (satisfy \eqref{eq:5.07d}), then 
\begin{align}\label{eq:rho_a_as}
\rho(t) & = \frac{t}{f(t)} \rho_\alpha(1)(1+\oo(1)), & \big(\,\rho(t) & = tg(t) \rho_\alpha(1)(1+\oo(1))\,\big)
\end{align}
as $t\to\infty$, where $\rho_\alpha$ is given by \eqref{eq:rho_alpha} and $f$ and $g$ are the functions of Definitions \ref{def:f_func} and  \ref{def:g_func}, respectively. 
\end{corollary}

\begin{proof}
The proof follows by combining Theorem \ref{th:} and Theorem \ref{cor:} with formula \eqref{eq:rho_alpha} and Theorem \ref{th:tauber}.
\end{proof}
 
\subsection{The case \texorpdfstring{$\alpha=\infty$}{alpha=infty}} \label{sec:5.3}

For a definition and basic facts about rapidly varying functions we refer the reader to Appendix \ref{app:osv}. The next theorem is the main result of this subsection.

\begin{theorem}\label{th:infty}
Let $H$ be a trace normed Hamiltonian on $[0,\infty)$ and $m$ be the corresponding $m$-function.  If the function $A$ in \eqref{eq:ABC} is rapidly varying at $0$, then 
\be\label{eq:5.47}
m(r\mu) = -\frac{1}{\mu f(r)}(1+\oo(1)),\quad r\to \infty,
\ee
where $f$ is the function from Definition \ref{def:f_func} and the estimate holds uniformly for $\mu$ in any compact subset of $\C$ not intersecting the real line.
\end{theorem}

\begin{proof}
Let $f$ and $F$ be the functions from Definition \ref{def:f_func}. Clearly, 
$F\downarrow 0$ and $f\uparrow \infty$ as $r\uparrow \infty$. 
Using the scaling \eqref{eq:5.10}, let us compute the expression  
\be\label{eq:5.48}
\int_0^{\xi_r(x)}H_r(t)dt=
\begin{pmatrix}
rf(r)A(\xi_r(x)F(r^2)) & rB(\xi_r(x)F(r^2)) \\
rB(\xi_r(x)F(r^2)) & 	F(r^2)^{-1}C(\xi_r(x)F(r^2))
\end{pmatrix},
\ee
in the limit when $r\to \infty$.
Here $\xi_r(x)$ is the inverse of 
\begin{align*}
\eta_r(x)=\int_0^x\tr\, H_r(t)dt & =  \frac{A(xF(r^2))}{A(F(r^2))} + \frac{ C(xF(r^2))}{F(r^2)}\\
 &= x + \frac{A(xF(r^2))}{A(F(r^2))}\left(1 - \frac{A(F(r^2))}{F(r^2)}\right).  
\end{align*}
Let $x\in (0,1)$. Since $A$ is rapidly varying, we get 
\[
\frac{A(F(r^2))}{F(r^2)} = \oo(F(r^2)^\gamma) = \oo(1),\quad r\to\infty.
\]
The latter holds for every $\gamma>0$. Therefore,  we get
\be\label{eq:estEta}
\lim_{r\to\infty}\eta_r(x)=
\begin{cases} 
x, & x\in(0,1),\\
2, & x=1,\\
\infty, & x>1.\end{cases}
\ee
 Moreover, $A$ is increasing on $(0,1)$ and hence $\eta_r(x)$ converges to $x$ locally uniformly on $[0,1)$. The latter in particular implies that $\xi_r(x)\to x$ locally uniformly on $[0,1)$ as $r\to\infty$. Since $B(x)\le A(x)$ for all $x>0$, we get
\be\label{eq:estB}
| rB(\xi_r(x)F(r^2))|\le r A(\xi_r(x)F(r^2)) = \frac{rf(r) A(\xi_r(x)F(r^2))}{f(r)}.
\ee
Therefore, we finally conclude that the integral in \eqref{eq:5.48} approaches $\bigl(\begin{smallmatrix} 0 & 0 \\ 0 & x \end{smallmatrix}\bigr)$ as $r\to \infty$ for all $x\in (0,1)$. 

Now consider $x>1$.
In view of \eqref{eq:estEta} we get $\xi_r(x)\to 1$ as $r\to\infty$.
Since 
\[
x=\xi_r(x) + \frac{A(\xi_r(x)F(r^2))}{A(F(r^2))} - \frac{A(\xi_r(x)F(r^2))}{F(r^2)} 
\]
and using the fact that $A$ is rapidly varying at $0$, we obtain
\begin{align*}
\lim_{r\to\infty}\frac{A(\xi_r(x)F(r^2))}{A(F(r^2))} & = x-1, & \lim_{r\to\infty}\frac{C(\xi_r(x)F(r^2))}{F(r^2)} & =1.
\end{align*}
whenever $x>1$. Moreover, using \eqref{eq:estB} we get
\[
|rB(\xi_r(x)F(r^2))| \le rA(\xi_r(x)F(r^2)) = \frac{(x-1)(1+\oo(1))}{f(r)} =\oo(1)
\]
as $r\to\infty$. 
Therefore, for all $x>1$ the limit in \eqref{eq:5.48} equals $\bigl(\begin{smallmatrix} x-1 & 0 \\ 0 & 1 \end{smallmatrix}\bigr)$.
Hence the limiting Hamiltonian $H_0$ is given by \eqref{eq:ex2}. It remains to note that the corresponding $m$-function is given by \eqref{eq:m_ex2}. Proposition \ref{prop:dB_corresp} completes the proof.
\end{proof}

\begin{remark}
\begin{enumerate}[label=\emph{(\roman*)}, ref=(\roman*), leftmargin=*, widest=iii]
\item In a similar way as Theorem \ref{cor:}, one can prove an analogue of Theorem \ref{th:infty} for the case when $C\in RV_\infty(0)$.
\item Using Lemma \ref{th:uniq}, one can extend Theorem \ref{th:infty} to the case of Hamiltonians which are not trace normed.
\end{enumerate}
\end{remark}

We conclude this subsection with the following result.

\begin{corollary}\label{cor:rho_infty}
Let $H$ be a trace normed Hamiltonian on $[0,\infty)$ and $\rho$ be the corresponding 
spectral function. If $A \in RV_\infty(0)$ ($C\in RV_\infty(0)$), then 
\begin{align}
\rho(t) - \rho(-t) & = \frac{t}{f(t)}\, (1+\oo(1)), & \ \big(\,\rho(t) - \rho(-t) & = tg(t)\, (1+\oo(1))\,\big),
\end{align}
as $ t\to \infty$, where $f$ and $g$ are the functions from Definition \ref{def:f_func} and \ref{def:g_func}, respectively.
\end{corollary}

The proof is analogous to the proof of Corollary \ref{cor:rho_a} and we omit it.

\subsection{General \texorpdfstring{$2\times 2$}{2x2} first order systems} \label{sec:5.4}
Let $H$ be a Hamiltonian on $[0,L)$ and $Q: [0,L)\rightarrow \R^{2\times 2}$ be locally integrable and so that $Q(x)=Q(x)^*$ for almost all $x\in [0,L)$. We consider the general first order spectral problem 
\be\label{eq:sp_gen}
JY'+QY=zHY
\ee 
on the interval $[0,L)$. 
For simplicity, let us assume that \eqref{eq:sp_gen} is in the limit point case at the right endpoint $L$. Although we considered the case when $Q\equiv 0$ in Section~\ref{sec:cansys}, one can introduce the $m$-function for \eqref{eq:sp_gen} in much the same way as it was done for the system \eqref{eq:II.01}. Under the above assumptions on the coefficients $Q$ and $H$, there is a fundamental matrix-solution $U(z,x)$ to \eqref{eq:sp_gen} satisfying the initial condition $U(z,0)=I_2$ and the $m$-function for \eqref{eq:sp_gen} is defined by \eqref{eq:II.06}.

In fact, the problem to investigate the spectral asymptotics for \eqref{eq:sp_gen} can be reduced to the case $Q\equiv 0$ (see, e.g., \cite[Chapter VI.1]{gk}, \cite{lema}). Assume for simplicity that $\det H(x)\neq 0$ for almost all $x\in [0,L)$. Let $\U(0,x)$ be the matrix solution of \eqref{eq:sp_gen} for $z=0$ and set 
\be\label{eq:ti_R}
\tilde{H}(x):=\U(0,x)^*H(x)\U(0,x),\quad x\in[0,L).
\ee
Define the map $\mathcal{V}: L^2([0,L),H(x))\to L^2([0,L),\tilde{H}(x))$ by
\be\label{eq:map}
(\mathcal{V}F)(x)=\tilde{F}(x):=\U(0,x)^{-1}F(x),\quad x\in[0,L).
\ee
Firstly, notice that $\mathcal{V}$ is an isometry. Indeed,
\begin{align}\begin{split}
\|\ti{F}\|^2_{L^2([0,L);{\tilde{H}})}&=\int_0^L\tilde{F}(x)^*\tilde{H}(x)\tilde{F}(x)dx\\
&=\int_0^LF(x)^*\U(0,x)^{-*}\tilde{H}(x)\U(0,x)^{-1}F(x)dx\\
&=\int_0^LF(x)^*H(x)F(x)dx=\|F\|^2_{L^2([0,L);{H})}.
\end{split}\end{align}
Further, noting that
\begin{align}
\U(z,x)^*J\U(z,x)=J,\quad z\in\C,
\end{align}
after straightforward calculations we find that $\tilde{\U}(z,x)=(\mathcal{V}\U)(z,x)$ is a matrix-solution to 
\be\label{eq:tilde_sp}
JY'=z\tilde{H}Y,
\ee
where $\ti{H}$ is given by \eqref{eq:ti_R}. 
It remains to note that the $m$-function for $\ti{H}$ defined by \eqref{eq:II.06} is the $m$-function for the system \eqref{eq:sp_gen}. Indeed, we have $\ti{\U}(z,x) = \U(0,x)^{-1}\U(z,x)$ and hence
\be
\tilde{\U}(z,x)\begin{pmatrix}
1\\
m(z)
\end{pmatrix}=\U(0,x)^{-1}\U(z,x)\begin{pmatrix}
1\\
m(z)
\end{pmatrix}. 
\ee
It remains to note that $\mathcal{V}$ is an isometry and hence the left-hand side belongs to $L^2([0,L);\ti{H})$ if and only if the right-hand side belongs to $L^2([0,L),H)$. 

\begin{theorem}\label{th:3.16}
Let $H$ be a Hamiltonian on $[0,L)$, $Q$ be a potential as above and $m$ be the corresponding $m$-function. 
Then the following conditions are equivalent:
\begin{enumerate}[label=\emph{(\roman*)}, ref=(\roman*), leftmargin=*, widest=iii]
\item For some $a_0\geq0$, $b_0\in\R$ and $c_0>0$ with $h_0^2=a_0c_0-b_0^2\geq 0$ one has \eqref{eq:5.21}.
\item For some $\zeta_0$ in the closure of $\C_+$ one has~\eqref{eq:5.02} 
as $z\rightarrow\infty$ in any closed sector in $\C_+$.
\end{enumerate}
In this case, the quantities $a_0$, $b_0$, $c_0$ and $\zeta_0$ are related as in~\eqref{eq:m_1} and~\eqref{eqn31rel}. 
\end{theorem}

\begin{proof}
Since $\U(0,\cdot)$ is continuous and $\U(0,0)=I_2$, it follows from \eqref{eq:ti_R} that 
\begin{align*}
 \left|\int_0^x \ti{H}(t) dt - \int_0^x H(t)dt \right| = \oo\left( \int_0^x |a(t)| + |b(t)| + |c(t)| dt\right)
\end{align*}
as $x\to 0$. Upon noting that from $H(x)\ge 0$ we have
\[
|b(x)|\le \sqrt{a(x)c(x)}\le \frac{a(x)+c(x)}{2} = \frac{1}{2}\tr\, H(x)
\]
for almost all $x\in [0,L)$, this implies $\tilde{\eta}(x) = \eta(x)(1+\oo(1))$ and furthermore
\begin{align*}
 \left| \frac{1}{\tilde{\eta}(x)} \int_0^x \ti{H}(t) dt - \frac{1}{\eta(x)} \int_0^x H(t)dt\right| \rightarrow 0
\end{align*}
as $x\to 0$. Hence $H$ satisfies \eqref{eq:5.21} if and only if $\ti{H}$ satisfies \eqref{eq:5.21} with the same constants $a_0$, $b_0$, $c_0$ and it remains to apply Corollary \ref{cor:4.2}.
\end{proof}

\begin{remark}
Under the additional assumptions $L=\infty$, $b\equiv 0$ on $[0,\infty)$ and 
\be\label{eq:5.31}
\int_0^{x} |a(t)-a_0|+|c(t)-c_0|dt=\oo(x), \quad x\to 0,
\ee
with some positive constants $a_0$, $c_0>0$, it was proven by W.\ N.\ Everitt, D.\ B.\ Hinton and J.\ K.\ Shaw (see Theorem 1 in \cite{ehs}) that 
\be
m(z)=\I\sqrt{\frac{a_0}{c_0}}+\oo(1),\quad z\to \infty 
\ee
uniformly in any sector in $\C_+$. 
A weaker version of the implication $(i)\Rightarrow (ii)$ in Theorem \ref{th:3.16} for non-diagonal Hamiltonians $H$ was established in \cite[Theorem 4.2]{wi}. 
 Also note that Theorem \ref{th:3.16} solves the problem posed in \cite{ehs} of characterizing non-diagonal matrix-functions $H$ such that the $m$-function of \eqref{eq:sp_gen} satisfies \eqref{eq:5.02}.
\end{remark}

Since $\U(0,x)\rightarrow I_2$ as $x\to 0$, in many cases of interest the leading term in the asymptotic expansion of $m$ does not depend on the potential $Q$. However, in general one has to be more careful (especially in the case when $A$ or $C$ are rapidly varying functions) when studying the asymptotic behavior of the Hamiltonian $\tilde{H}$ given by \eqref{eq:ti_R}.

\section{Spectral asymptotics for radial Dirac operators}\label{sec:04}

\subsection{The (singular) Weyl function}
The radial Dirac operator (see \cite[Chapter~4]{th}) is a self-adjoint extension of the differential expression
\begin{equation}
\label{eq:radexprB}
\tau = \frac{1}{\I} \sigma_2 \frac{d}{dx} + \frac{\kappa}{x}\sigma_1+ Q(x)
\end{equation}
considered in the Hilbert space $L^2([0,\infty))^2$. Here 
\begin{align}
\sig_1 & =\begin{pmatrix} 0 & 1 \\ 1 & 0\end{pmatrix}, &
\sig_2 & =\begin{pmatrix} 0 & -\I \\ \I & 0\end{pmatrix}, &
\sig_3 & =\begin{pmatrix} 1 & 0 \\ 0 & -1\end{pmatrix},
\end{align}
are the Pauli matrices and the potential $Q$ is a symmetric matrix function given by (we do not include a mass as it can be absorbed in $q_{sc}$)
\begin{equation}\label{eq:q_gen}
Q(x)= q_{el}(x)I_2 + q_{am}(x)\sigma_1 + q_{sc}(x)\sigma_3
=\begin{pmatrix}q_{sc}(x) & q_{am}(x) \\ q_{am}(x)  &-q_{sc}(x)\end{pmatrix},
\end{equation}
with $q_{\rm el}$, $q_{\rm sc}$, $q_{\rm am} \in L^1_{\loc}([0,\infty),\R)$ and $\kappa\in\R$. 
 
Note that $\tau$ is singular at the left endpoint $0$ if $\kappa\neq 0$ as $q_{12}=\kappa/x$ is not integrable near this 
endpoint and also singular at the right endpoint $\infty$ because the endpoint itself is infinite. Moreover, $\tau$ is limit point at $\infty$. 

In the case $Q\equiv 0$ the underlying differential equation $\tau y = z y$ can be solved in terms of Bessel functions and
in the general case standard perturbation arguments can be used to show the following (see \cite[App.~A]{ahm2} and \cite[Section~8]{bekt}):

\begin{lemma}
\label{lemPRDPhi}
Suppose that $\kappa\ge 0$ and $\kappa+\frac{1}{2}\notin \N$. The equation $\tau y = z y$ with $\tau$ given by \eqref{eq:radexprB} has a system of solutions $\Phi(z,x)$ and $\Theta(z,x)$ which is real entire with respect to $z\in\C$ for every $x\in (0,\infty)$ and satisfies
\begin{align}\label{eq:estphi}
\Phi(z,x) & = C_\kappa x^\kappa\, \widetilde{\Phi}(z,x), & \Theta(z,x) & =C_{\kappa}^{-1}x^{-\kappa}\,\widetilde{\Theta}(z,x),
\end{align}
\begin{align}
\widetilde{\Phi}(z,\redot),\ \widetilde{\Theta}(z,\redot) & \in C([0,\infty)), & \widetilde{\Phi}(z,0) & =\begin{pmatrix}0\\ 1\end{pmatrix}, &  
\widetilde{\Theta}(z,0) & =\begin{pmatrix}1\\ 0\end{pmatrix}.
\end{align}
Here
\be\label{eq:ckappa}
\quad C_\kappa:=\frac{\sqrt{\pi}}{2^\kappa\Gamma(\kappa+1/2)}.
\ee
Moreover, $x^{-1}\widetilde{\Phi}_1(z,x)$ and $x^{-1}\widetilde{\Theta}_2(z,x) \in L^1_{\loc}([0,\infty))$.
\end{lemma}

Let $\kappa\ge 0$ and $\Phi$, $\Theta$ be the system of fundamental solutions as in Lemma \ref{lemPRDPhi}. The function $m:\C\backslash\R\to \C$ such that
\be\label{eq:radMfun}
\Psi(z,\redot) = \Theta(z,\redot) + m(z) \Phi(z,\redot) \in L^2((1,\infty);\C^2)
\ee
will be called {\em the Weyl function} of the radial Dirac operator \eqref{eq:radexprB}. It is known (see \cite[Theorem 4.5]{bekt14}) that $m$ admits the integral representation
\be\label{Mir}
m(z) = g(z) + (1+z^2)^{\floor{\kappa+\frac{1}{2}}} \int_\R \left(\frac{1}{t-z} - \frac{t}{1+t^2}\right) \frac{d\rho(t)}{(1+t^2)^{\floor{\kappa+\frac{1}{2}}}},
\quad z\in\C\backslash\R,
\ee
where $g$ is a real entire function and the nondecreasing function $\rho$, called the spectral function, satisfies 
\begin{align}
\int_\R \frac{d\rho(t)}{(1+t^2)^{\floor{\kappa+\frac{1}{2}}}} & =\infty, & \int_\R \frac{d\rho(t)}{(1+t^2)^{\floor{\kappa+\frac{1}{2}}+1}} & <\infty.
\end{align}

Let us mention that all the above considerations can be easily extended to the case $\kappa<0$. Indeed, note that 
\begin{align}
{\tau}_-:=\sigma_2\tau \sigma_2 = \frac{1}{\I} \sigma_2 \frac{d}{dx} -\frac{\kappa}{x}\sigma_1 + {Q}_-(x),
\end{align}
 where the potentials $Q$ and ${Q}_-$ are connected by
\begin{align}
{Q}_-(x) = \sigma_2Q(x)\sigma_2 = q_{el}(x) I_2 - q_{am}(x)\sigma_1 -  q_{sc}(x)\sigma_3.
\end{align}
If $\kappa<-\frac{1}{2}$, then the corresponding system of fundamental solutions is given by $\Phi_-(z,x)={\I}\sig_2\Phi(z,x)$ and $\Theta_-(z,x)={\I}\sig_2\Theta(z,x)$. Note that we again have 
\begin{align}
W(\Theta_-(z),\Phi_-(z))=W({\I}\sig_2\Theta(z),{\I}\sig_2\Phi(z))=W(\Theta(z),\Phi(z))=1.
\end{align}
Moreover, the corresponding Weyl function $m_-$ coincides with the Weyl function $m$ defined by \eqref{eq:radMfun}.

In the case $\kappa\in (-\frac{1}{2},0)$, we choose another fundamental system of solutions:
\begin{align}\label{eq:FSol-}
\Phi_-(z,x) & ={\I}\cos(\pi \kappa)\sig_2\Theta(z,x), & \Theta_-(z,x) & =\frac{1}{\I\cos(\pi\kappa)}\sig_2\Phi(z,x)
\end{align}
and hence the corresponding  $m$-functions $m$ and ${m}_-$ are related via
\be\label{eq:tim=m}
{m}_-(z)=-\frac{1}{\cos^2(\pi \kappa)\,m(z)},\quad z\in\C\backslash\R.
\ee
Therefore, it suffices to investigate the case $\kappa\ge 0$.

Finally, let us mention that in the case $Q\equiv 0$ and $\kappa+\frac{1}{2}\notin\Z$ one of the Weyl functions and the corresponding spectral function are given by (see \cite[Section 4]{bekt})
\begin{align}\label{eq:unpMrho}
M_\kappa(z) & = \begin{cases}
 \frac{-1}{\cos(\pi\kappa)}\frac{(-z^2)^{\kappa+\frac{1}{2}}}{z}, & \kappa>-\frac{1}{2},\\
 \frac{-1}{\cos(\pi\kappa)}\frac{(-z^2)^{|\kappa|+\frac{1}{2}}}{z}, & \kappa<-\frac{1}{2},
\end{cases},
& \rho_\kappa(t) & = \begin{cases}  \frac{{\rm sign}(t)}{2\kappa+1}|t|^{2\kappa+1}, & \kappa>-\frac{1}{2},\\
\frac{{\rm sign}(t)}{2|\kappa|+1}|t|^{2|\kappa|+1}, & \kappa<-\frac{1}{2}.
\end{cases}
\end{align}

\subsection{The case \texorpdfstring{$|\kappa|<\frac{1}{2}$}{|kappa|<1/2}} We begin with the case
 when $\tau$ is limit circle at $x=0$. Moreover, as pointed out in the previous subsection, it suffices to consider the case when $\kappa\in [0,\frac{1}{2})$. Let $\U(x)=\U(0,x)$ be the fundamental matrix solution  of $\tau y=0$ given by 
\be
\U(x)=\begin{pmatrix}
\theta_1(0,x) & \phi_1(0,x)\\
\theta_2(0,x) & \phi_2(0,x)
\end{pmatrix}.
\ee
Due to our assumption on $\kappa$, $\U^*\U\in L^1((0,a),\C^{2\times 2})$ for all $a>0$. Therefore, we can apply the gauge transformation \eqref{eq:map} in order to transform \eqref{eq:radexprB} to the canonical form \eqref{eq:II.01}. More precisely, the gauge transformation \eqref{eq:map} transforms $\tau y = zy$  to the canonical differential expression 
\begin{align}\label{eq:transf}
J Y' & =z HY, & H(x) & =\U^*(x)\U(x),\quad x\in [0,\infty).
\end{align}
Moreover, it is straightforward to check that the $m$-function for \eqref{eq:transf} defined by \eqref{eq:II.06} coincides with the Weyl function defined by \eqref{eq:radMfun}.

Next notice that by Lemma \ref{lemPRDPhi} we have 
\be\label{eq:M_0}
\U(x)\sim \begin{pmatrix}
C_\kappa^{-1}x^{-\kappa} & o(x^{\kappa})\\
o(x^{-\kappa}) & C_\kappa x^{\kappa}
\end{pmatrix},\quad x\to 0,
\ee
and thus
\be\label{eq:as_H}
{H}(x)\sim \begin{pmatrix}
C_\kappa^{-2}x^{-2\kappa} & o(1)\\
o(1) & C_\kappa^2x^{2\kappa}
\end{pmatrix},\quad x\to 0.
\ee

Applying Theorem \ref{th:} to \eqref{eq:transf}, we can easily obtain the high-energy asymptotic behavior of the $m$-function associated with \eqref{eq:radexprB}. Namely, define 
\[
\eta'(x) = \tr\, {H}(x),\quad x\in[0,\infty),
\]
and set 
\be\label{eq:as_H2}
\tilde{{H}}(x) = \frac{1}{\eta'(\eta^{-1}(x))}{H}(\eta^{-1}(x)),\quad x\in[0,\infty),
\ee
where $\eta^{-1}$ is the inverse function of $\eta$. The Hamiltonian $\tilde{{H}}$ is trace normed and it remains to find the asymptotic behavior of its elements near $x=0$. Noting that
\be
\int_0^x \tilde{{H}}(t) dt = \int_0^{\eta^{-1}(x)}{H}(t)dt
\ee
and 
\be\label{eq:eta02}
\eta^{-1}(x)= (1+\oo(1))\begin{cases}(C_\kappa^2(1-2\kappa)\, x)^{\frac{1}{1-2\kappa}},& \kappa\in(0,1/2),\\
\quad \frac{1}{2}x,&\kappa=0,
\end{cases}
\ee
as $x\to 0$, we get from \eqref{eq:as_H} and \eqref{eq:as_H2} by applying Lemma \ref{lem:a.asymp} that $\tilde{{H}}(x) \rightarrow \frac{1}{2}I_2$ as $x\to 0$ for $\kappa=0$ and 
\be
\int_0^x\tilde{{H}}(t)dt = \begin{pmatrix}
1+\oo(1) & \oo (x^{\frac{1}{1-2\kappa}}) \\
\oo (x^{\frac{1}{1-2\kappa}})  & C_\kappa^{\frac{4}{1-2\kappa}} (1+2\kappa)^{-1}(1-2\kappa)^{\frac{1+2\kappa}{1-2\kappa}}x^{\frac{1+2\kappa}{1-2\kappa}}(1+\oo(1))
\end{pmatrix},
\ee
as $x\to 0$ whenever $\kappa\in (0,1/2)$. 
It remains to apply Corollary \ref{cor:3.08B} with $\alpha = \frac{4\kappa}{1-2\kappa}$. After straightforward calculations we find that
\be\label{eq:4.12}
m(z) =  \frac{-1}{\sin((\kappa+\frac{1}{2})\pi)}\frac{(-z^2)^{\kappa+\frac{1}{2}}}{z}(1+\oo(1)), 
\ee
as $z\to \infty$ in any sector in $\C_+$. Furthermore, Theorem \ref{th:tauber} implies that the corresponding spectral function satisfies
\be
\rho(t) =  \frac{1}{\pi (2\kappa+1)} {\rm sign}(t)|t|^{2\kappa+1}(1+\oo(1)),\quad |t|\to \infty.
\ee
Moreover, in view of \eqref{eq:FSol-} and \eqref{eq:tim=m}, \eqref{eq:4.12} remains true for $\kappa\in (-\frac{1}{2},0)$.

\subsection{The case \texorpdfstring{$\kappa+\frac{1}{2}\notin\Z$}{kappa+1/2 not in Z}}
Now assume that $\kappa>\frac{1}{2}$. Then in view of Theorem 4.5 and Remark 4.7 from \cite{bekt14}, one of the Weyl functions $m$ of the radial Dirac operator \eqref{eq:radexprB} has the form  
\be\label{Msum}
M(z)=z^{2\floor{\kappa+\frac{1}{2}}} M_0(z)
\ee
as $z\to \infty$ in any nonreal sector in $\C$. Here $M_0$ is a Weyl function of the radial Dirac operator with the angular momentum $\kappa_0=\kappa - \floor{\kappa+\frac{1}{2}} \in (-\frac{1}{2},\frac{1}{2})$. 

Summarizing, we get the following result.

\begin{theorem}\label{th:dir_bess}
Let $\tau$ be the radial Dirac differential expression \eqref{eq:radexprB}--\eqref{eq:q_gen} with $\kappa\in\R$, $\kappa+\frac{1}{2}\notin \Z$ and $q_{el}$, $q_{sc}$, $q_{am}\in L^1_{\loc}([0,\infty),\R)$. Then there is a real entire function $g$ such that in any nonreal sector the Weyl function  defined by \eqref{eq:radMfun} satisfies
\be
m(z) - g(z)= M_\kappa(z)(1+\oo(1)),\quad |z|\to \infty.
\ee 
Moreover, the corresponding spectral function satisfies 
\be
\rho(t)= \rho_\kappa(t)(1+\oo(1)),\quad |t|\to \infty.
\ee 
Here $M_\kappa$ and $\rho_\kappa$ are given by \eqref{eq:unpMrho}. 
\end{theorem} 

\begin{remark}
If $\kappa+\frac{1}{2}\in \N$, then in the unperturbed case, the singular $m$-function is given by (see \cite[Section 4]{bekt})
\begin{align*}
M_\kappa(z) = - \frac{1}{\pi}z^{2\kappa} \log(-z^2),\quad z\in\C\backslash\R. 
\end{align*}
The function $M_{\kappa}$ is not a Herglotz--Nevanlinna function and hence we cannot apply our approach directly. Using the gauge transformation \eqref{eq:map}, we can transform the equation $\tau y = zy$ into the first order system 
\be
JY' = z HY,\quad x\in (0,\infty),\qquad H(x) = \begin{pmatrix} x^{-\kappa} & 0 \\ 0 & x^{\kappa} \end{pmatrix},
\ee
which is limit point at both endpoints, at $x=0$ and $x=\infty$. The approach developed in Section \ref{sec:03} can be extended to this case by using the extension of de Branges theory developed by M.\ Langer and H.\ Woracek in \cite{lw1,lw2}. 
\end{remark}

\subsection{Spectral asymptotics for radial Schr\"odinger operators}\label{sec:05}

There is a straightforward connection with the standard theory for one-dimensional Schr\"odinger operators
(cf.\ \cite{kst2}) if our Dirac operator is supersymmetric, that is, $q_{\rm el}= q_{\rm sc}=0$. In this case, a straightforward computation verifies that 
\be\label{eqhsq}
\tau^2 = \begin{pmatrix}
a_q^- a_q^+  & 0\\ 0 & a_q^+a_q^- 
\end{pmatrix},\qquad a_q^\pm  = \pm \frac{d}{dx} + \frac{\kappa}{x} + q_{\rm am}(x).
\ee
Here, $a_q^+ a_q^-$ and $a_q^-a_q^+$ are generalized one-dimensional Schr\"odinger differential expressions
of the type considered in \cite{ahm,egnt1,egnt2}. If $q_{am}\in AC_{\loc}([0,\infty))$, then
\begin{align}
a_q^- a_q^+ & =-\frac{d^2}{dx^2}+\frac{\kappa(\kappa+1)}{x^2}+q_+, & a_q^+ a_q^- & =-\frac{d^2}{dx^2}+\frac{\kappa(\kappa-1)}{x^2}+q_-,
\end{align}
are the usual radial Schr\"odinger differential expressions where
\be\label{eq:q_pm}
q_\pm(x)=\pm\frac{2\kappa}{x}q_{\rm am}(x)+q_{\rm am}(x)^2\mp q_{\rm am}'(x)\in L^1_{\loc}([0,\infty)).
\ee
If $q_{\rm am}\in L^2_{\loc}([0,\infty))$, then \eqref{eq:q_pm} implies that $q_\pm$ is a $W^{-1,2}_{\loc}([0,\infty))$ distribution. This class of potentials received a lot of attention during the last 15 years (see, e.g., \cite{ahm, sash99}).  However, if $q\notin L^2_{\loc}([0,\infty))$, then $q_\pm$ need not be a distribution and for further details concerning this class of Sturm--Liouville operators, we refer to \cite{ahm, egnt2}. In particular, the singular Weyl--Titchmarsh theory was developed in \cite{egnt2}. 

Let  $\Phi(z,x)$ and $\Theta(z,x)$ be entire solutions to $\tau u=z u$ as in Lemma \ref{lemPRDPhi}. Then $\phi(z^2,x)=\Phi_1(z,x)$ and $\theta(z^2,x)=\Theta_1(z,x)$ are entire solutions to $a_q^- a_q^+ y=z^2 y$ satisfying $\theta(z,x)\phi'(z,x) - \theta(z,x)\phi'(z,x)\equiv 1$.  The singular Weyl function $m_q:\C\backslash\R\to \C$ for $a_q a_q^*$ is defined by the requirement
\be\label{eq:m_q}
\psi(z,\redot) = \theta(z,\redot) + m_q(z)\phi(z,\redot) \in L^2(1,\infty). 
\ee
It follows from \cite[Section 3]{bekt} that $m_q$ is closely connected with singular $m$-function \eqref{eq:radMfun} of the radial Dirac operator:
\be
m(z)=\frac{m_q(z^2)}{z},\quad z\in\C\backslash\R.
\ee
Combining this result with Theorem \ref{th:dir_bess}, we immediately obtain the high-energy asymptotics for the radial Schr\"odinger operator defined by $a_q^- a_q^+$ in $L^2(\R_+)$. Namely, there is a real entire function $\tilde{g}$ such that 
\be\label{eq:mqass}
m_q(z) - \tilde{g}(z) = \frac{-1}{\cos(\pi\kappa)} (-z)^{\kappa+\frac{1}{2}} (1+\oo(1)),\quad |z|\to \infty,
\ee
in any nonreal sector. 
Moreover, the corresponding spectral measure satisfies
\be\label{eq:rho_qass}
\rho_q(t) = \frac{t^{\kappa+\frac{3}{2}}}{\pi(\kappa+\frac{3}{2})}(1+\oo(1)),\quad t\to+\infty.
\ee

\begin{remark}
For $|\kappa|<\frac{1}{2}$ and $q_{\rm am} \in BV_{\loc}([0,\infty)$, formula \eqref{eq:rho_qass} was first announced in \cite{ka56}. Using a different approach, \eqref{eq:mqass} was established in \cite{kt2} (see also \cite{kst3, kt}) under the assumptions $q_{\rm am}\in AC_{\loc}(0,c)$ and $xq_{\rm am}'\in L^1(0,c)$ for some $c>0$. 
\end{remark}

\section{Spectral asymptotics for Krein's strings}\label{sec:06}

The main object of this section is the Krein string spectral problem of the form  
\be\label{eq:krein}
-y'' = z\, y\, \omega
\ee
 on an interval $[0,L)$ with $L\in (0,\infty]$ and a non-negative Borel measure $\omega$ on $[0,L)$. 
If $\int_{[0,L)}x^2d\omega(x)<\infty$, then we will impose a Dirichlet-type boundary condition at the right endpoint $L$. When the measure $\omega$ is locally absolutely continuous with respect to the Lebesgue measure on $[0,L)$, then the spectral problem \eqref{eq:krein} is equivalent to the usual Sturm--Liouville spectral problem, $-y''=z\, \Wr' y$, where we set $\Wr(x)=\omega([0,x))$ for all $x\in [0,L)$. 
If $\omega$ has a nontrivial singular component, then there are at least two ways to treat \eqref{eq:krein}: by using integral equations \cite{kk74} or by employing quasi-derivatives \cite{sash99}. In fact, both approaches lead to the same result and for further details we refer to \cite{CHPencil, IndString, MeasureSL}. 

To define the $m$-function for \eqref{eq:krein} we first introduce the fundamental system of solutions $\theta(z,x)$ and $\phi(z,x)$ of \eqref{eq:krein} with the initial conditions 
\begin{align}
\phi(z,0) & = \theta'(z,0-)=0, & \phi'(z,0-) & =\theta(z,0)=1.
\end{align}
It is known \cite{kk74} that such a system exists under the above assumptions on $\omega$ and  that $\phi(z,x)$ as well as $\theta(z,x)$ are entire functions in $z$ for every $x\in [0,L)$. 
 The Weyl function (or the Dirichlet $m$-function) $m_D$ of the string \eqref{eq:krein} is defined as the unique function $m_D:\C\backslash[0,\infty)\to\C$ such that the solution 
\be\label{eq:mDstring}
\psi(z,\redot) = \theta(z,\redot) + m_D(z)\phi(z,\redot) 
\ee
belongs to $L^2([0,L);\omega)$ and satisfies $\psi(z,L-)=0$ if $\int_0^Lx^2d\omega(x)<\infty$ for every $z\in\C\backslash[0,\infty)$. 
It is indeed a Herglotz--Nevanlinna function belonging to the Stieltjes class $(S^{-1})$, that is, it admits the integral representation
\be
m_D(z) = c_1z - c_2 +\int_{(0,\infty)} \left(\frac{1}{t-z} -\frac{1}{t} \right)d\rho_D(t),\quad z\in\C\backslash[0,\infty),
\ee
for some nonnegative constants $c_1$, $c_2\in[0,\infty)$ and a non-decreasing function $\rho_D$ on $(0,\infty)$ for which the integral  
\begin{align}
 \int_{(0,\infty)}\frac{d\rho_D(t)}{t+t^2}
\end{align}
is finite. The function $\rho_D$ is called the spectral function of the string \eqref{eq:krein}.

It is well known that \eqref{eq:krein} can be transformed into a canonical system \eqref{eq:II.01} with a trace normed Hamiltonian (see \cite{gk}). Indeed, set 
\be\label{eq:6.03}
s(x) =x+\Wr(x),\quad x\in [0,L),
\ee 
and denote by $\x$ its generalized inverse, $\x:[0,\infty)\to [0,L]$, defined by
\be\label{eq:6.04}
\x(s) =\begin{cases}
\min\{x\, | \, s(x)=s\}, & s\in [0,s(L-)),\\
L, & s\ge s(L-).
\end{cases}
\ee
The function $\x$ is locally absolutely continuous with $0\le \x'\le 1$ almost everywhere on $[0,\infty)$. With these definitions, the Weyl function $m$ for the canonical system 
\begin{align}\label{eq:cs_kre}
JY' & = z HY, & H(s) & =\begin{pmatrix} 1 - \x'(s) & 0 \\ 0 & \x'(s) \end{pmatrix}.
\end{align}
 is related to $m_D$ via  (see, e.g., \cite[Theorem 4.2]{kww})
\be\label{eq:m=MN}
m(z) z = m_D(z^2),\quad z\in \C_+.
\ee
Using this connection between the $m$-functions and the results from Section \ref{sec:03}, we can easily recover the main results of I.\ S.\ Kac \cite{ka71}, Y.\ Kasahara \cite{ka75} and C.\ Bennewitz \cite{ben} on the high-energy behavior of $m_D$. 

\begin{theorem}[\cite{ka75}]\label{th:askrein}
Let $\omega$ be  a non-negative Borel measure on $[0,L)$ and $m_D$ be the corresponding $m$-function. 
If we denote with $\tilde{f}:[0,\infty)\to[0,\infty)$ the generalized inverse of $x\mapsto 1/(x\Wr(x))$, then the following are equivalent:
\begin{enumerate}[label=\emph{(\roman*)}, ref=(\roman*), leftmargin=*, widest=iii]
\item We have $\Wr\in RV_{1+\alpha}(0)$ with some $\alpha\in (-1,\infty)$. 
\item With $d_{\nu}$ given by \eqref{eq:kac}, we have 
\be
m_D(r\mu) = -d_{\frac{1}{2+\alpha}} (-\mu)^{\frac{1}{2+\alpha}} \tilde{f}(r)^{-1} (1+\oo(1))
\ee
as $r\to\infty$, uniformly for all $\mu$ in compact subsets of $\C_+$.
\end{enumerate}
\end{theorem}

\begin{proof}
We will only consider the case when $\alpha \in (-1,0)$ since the other cases can be proved similarly. Then the function $s$ defined by \eqref{eq:6.03} is also regularly varying at zero with index $1+\alpha\in (0,1)$ and hence, by \cite[Theorem 1.5.12]{bgt}, $\x\in RV_{1+\tilde{\alpha}}(0)$ with $\tilde{\alpha} := \frac{-\alpha}{1+\alpha}\in (0,\infty)$. Therefore, we can apply Theorem \ref{cor:} and thus the $m$-function \eqref{eq:m=MN} satisfies
\[
m(r\mu) = rG(r^2) m_{-\tilde{\alpha}}(\mu)(1+\oo(1)),\quad r\to\infty,
\]
uniformly for all $\mu$ in compact subsets of $\C_+$. Here $G$ is the inverse of $1/(s\x(s))$ and $m_{-\tilde{\alpha}}$ is given by \eqref{eq:m_alpha2}. Using \eqref{eq:m=MN}, we get
\[
m_D(r\mu) = -\widetilde{m}_{-\tilde{\alpha}}(\mu)^{-1} rG(r)(1+\oo(1)),\quad r\to\infty,
\]
where $\widetilde{m}_{-\tilde{\alpha}}$ is given by \eqref{eq:m_alpha2C}. 
It remains to note that $(1+\tilde{\alpha})/(2+\tilde{\alpha}) = 1/(2+\alpha)$ and, moreover, by Lemma \ref{lem:a.asymp}
\[
rG(r) = \tilde{f}(r)^{-1}(1+\oo(1))
\]
as $r\to \infty$. Indeed, by Definition \ref{def:g_func}, $\x\circ G=1/(xG)$. On the other hand, since $\Wr \in RV_{1+\alpha}(0)$ with $\alpha\in(-1,0)$, Lemma \ref{lem:a.asymp} implies that $\x(r) = \Wr^{-1}(r)(1+\oo(1))$ as $r\to\infty$. Define the function $\widetilde{G}$ as the generalized inverse of $1/(x\Wr^{-1})$. Applying Lemma \ref{lem:a.asymp} once again, we see that $G(r)=\widetilde{G}(r)(1+\oo(1))$ as $r\to \infty$. Moreover, 
\[
\tilde{f} = \Wr^{-1}\circ \widetilde{G} \sim \x\circ\widetilde{G} \sim \x\circ G = \frac{1}{rG}
\] 
as $r\to\infty$.
\end{proof}

\begin{corollary}[\cite{ben, ka73, ka75}]\label{cor:askac}
Let $\omega$ be  a non-negative Borel measure on $[0,L)$ and $m_D$ be the corresponding $m$-function. 
Then, for given $\alpha\in(-1,\infty)$, there is a constant $C>0$ such that
\be\label{eq:6.10}
m_D(z) =  {-C}{d_{\frac{1}{2+\alpha}}} (-z)^{\frac{1}{2+\alpha}}(1+\oo(1)) 
\ee
as $z\to \infty$ uniformly in any nonreal sector if and only if 
\be\label{eq:6.11}
\Wr(x) = {C^{2+\alpha}} x^{1+\alpha}(1+\oo(1)), \quad x\to 0.
\ee
\end{corollary}

\begin{proof}
We only need to mention that, by Lemma \ref{lem:a.asymp}, $\tilde{f}(r) = (C^{2+\alpha}r)^{-\frac{1}{2+\alpha}}(1+\oo(1))$ as $r\to \infty$ if and only if $\Wr$ satisfies \eqref{eq:6.11}.
\end{proof}

\begin{remark}
It was first noted by I.\ S.\ Kac \cite{ka56, ka71,ka72, ka73} (see also \cite[\S 11.6]{kk74}) that the behavior of the spectral function $\rho_D$ of the  string $\eqref{eq:krein}$ at $\infty$ depends on the behavior of its mass distribution function $\Wr$ at $x=0$. Moreover, in \cite{ka73}, I.\ S.\ Kac proved the if part of Corollary \ref{cor:askac} and also found some necessary conditions on $\Wr$ for \eqref{eq:6.11} to hold. The only if part of Theorem \ref{th:askrein} was established independently by Y.\ Kasahara \cite{ka75} and C.\ Bennewitz \cite{ben}. In particular, the proof of Y.\ Kasahara in \cite{ka75} was based on the use of scaling and continuity properties of the Krein correspondence for strings. 
\end{remark}

\section{Spectral asymptotics for generalized indefinite strings}\label{sec:07}

Fix some $L\in (0,\infty]$, let $\omega \in H^{-1}_{\loc}([0,L))$ and $\dip$ be a non-negative Borel measure on $[0,L)$. The main object of this section is the spectral problem of the form 
\be\label{eq:indkrein}
-y'' = z\, y\,\omega + z^2\, y\,\dip
\ee
on the interval $[0,L)$. 
For a detailed discussion of the meaning of this spectral problem we refer to \cite{IndString}. In order to define the $m$-function for \eqref{eq:indkrein} we first introduce the fundamental system of solutions $\theta(z,x)$ and $\phi(z,x)$ of \eqref{eq:indkrein} with the initial conditions
\begin{align}
\phi(z,0) & = \theta'(z,0-)=0, & \phi'(z,0-) & =\theta(z,0)=1.
\end{align}
We define the Weyl function $M$ for the problem \eqref{eq:indkrein} by requiring that the solution  
 \begin{align}\label{eq:m_pm}
   \psi(z,\redot) = \theta(z,\redot) + M(z)\,z\phi(z,\redot)  
 \end{align}
 lies in $L^2([0,L);\dip)$ and that $\psi'(z,\redot)$ is square integrable on $[0,L)$ for every $z\in\C\backslash\R$.

Let us denote with $\Wr\in L^2_{\loc}([0,L))$ {\em the normalized anti-derivative} of  $\omega$, that is, 
 \begin{align}
    \omega(h) = - \int_0^L \Wr(x)h'(x)dx, \quad h\in H^1_{\mathrm{c}}([0,L)).
 \end{align} 
Note that if $\omega$ is a Borel measure, then $\Wr$ is simply given by the distribution function $\Wr(x)=\omega([0,x))$ for almost all $x\in [0,L)$.
Next define the following function
\begin{equation}\label{eq:7.08}
s(x) = x+\int_0^x \Wr^2(t)dt +\int_{[0,x)}d\dip(t), \quad x\in[0,L),
\end{equation}
as well as its generalized inverse $\x$ on $[0,\infty)$ via 
  \begin{align}\label{eqnIPxitrans}
   \x(s) = \sup\left\lbrace x\in[0,L)\,|\, s(x)\leq s\right\rbrace, \quad s\in[0,\infty). 
  \end{align}
  Let us point out explicitly that $\x(s)=L$ for $s\in[s(L),\infty)$ provided that $s(L)$ is finite.
  On the other side, if $s(L)$ is not finite, then $\x(s)<L$ for every $s\in[0,\infty)$ but $\x(s)\rightarrow L$ as $s\rightarrow\infty$. 
  Since $s$ is strictly increasing, we infer that $\x$ is non-decreasing and satisfies
  \begin{align}\label{eqnXiSigma}
   \x \circ s(x) & = x, \quad x\in[0,L); &  s\circ \x(s) & = \begin{cases} s, & s\in \ran{s},\\ \sup\lbrace t\in\ran{s}\,|\, t\leq s\rbrace, & s\not\in\ran{s}. \end{cases} 
  \end{align}
  Moreover, the function $\x$ is locally absolutely continuous with $0\leq \x'\leq 1$ almost everywhere on $[0,\infty)$. 

 With these definitions, the matrix function $H$ on $[0,\infty)$ given by 
  \begin{align}\label{eqnHamequ}
   H(s) = \begin{pmatrix} 1-\x'(s) & \x'(s)\Wr(\x(s)) \\ \x'(s)\Wr(\x(s)) & \x'(s) \end{pmatrix}, \quad s\in[0,\infty),
  \end{align}
  turns out to be a Hamiltonian and the $m$-function $m$ defined by \eqref{eq:II.06} of the corresponding canonical system is connected with the Weyl function $M$ of the spectral problem \eqref{eq:indkrein} via (for details we refer to the proof of Theorem 6.1 in \cite{IndString})
  \be\label{eq:M=-m}
  M(z) = -m(-z),\quad z\in \C\backslash\R. 
  \ee  

\begin{theorem}\label{th:7.1}
 Let $\omega \in H^{-1}_{\loc}([0,L))$, $\dip$ be a non-negative Borel measure on $[0,L)$ and $M$ be the corresponding Weyl function. Then the following conditions are equivalent:
 \begin{enumerate}[label=\emph{(\roman*)}, ref=(\roman*), leftmargin=*, widest=iii]
 \item For some $a_0\in[0,1)$ and $b_0\in\R$ with $h_0^2 = a_0(1-a_0)-b_0^2\geq0$ one has 
 \begin{align}\label{eq:7.16}
  s(x) \sim \frac{x}{1-a_0} \quad \text{and} \quad  \int_0^{x}\Wr(t)dt  \sim \frac{-b_0}{1-a_0} x, \qquad x\to 0. 
 \end{align}
 \item For some $\xi_0$ in the closure of $\C_+$ one has  
 \be\label{eq:7.15}
 M(z) \rightarrow \zeta_0 
\ee 
 as $ z\to \infty$ in any closed sector in $\C_+$.
\end{enumerate}
 In this case, the quantities $a_0$, $b_0$ and $\zeta_0$ are related as in~\eqref{eq:m_1} and~\eqref{eqn31rel}.
\end{theorem}

\begin{proof}
The proof immediately follows from Theorem \ref{th:4.1}. Indeed, by Theorem \ref{th:4.1}, \eqref{eq:7.15} is equivalent to
\begin{align}\label{eq:7.17}
{\rm x}(s) & = (1-a_0)s +\oo(s), & \int_0^{{\rm x}(s)}\Wr(t)dt & = -b_0 s +\oo(s),
\end{align}
as $s\to 0$. Applying Lemma \ref{lem:a.asymp} and noting that $1-a_0\not=0$, we conclude that $s(x) = \frac{x}{1-a_0}+\oo(x)$ as $x\to 0$. Making a change of variables in the second relation in \eqref{eq:7.17} and using the asymptotic behavior of $s(x)$ at $0$, we arrive at \eqref{eq:7.16}. 
\end{proof}

\begin{theorem}\label{th:7.2}
 Let $\omega \in H^{-1}_{\loc}([0,L))$, $\dip$ be a non-negative Borel measure on $[0,L)$ and $M$ be the corresponding Weyl function. Then, for given $\alpha>0$ and $D>0$, the following conditions are equivalent: 
 \begin{enumerate}[label=\emph{(\roman*)}, ref=(\roman*), leftmargin=*, widest=iii]
 \item We have 
  \be\label{eq:7.19}
   s(x) = D^{\frac{2+\alpha}{1+\alpha}}{x}^{\frac{1}{1+\alpha}} (1+\oo(1))\quad \text{and}\quad \int_0^{x}\Wr(t)dt = \oo(x^{\frac{2+\alpha}{2+2\alpha}}), \qquad x\to 0.
  \ee
 \item With the function $m_\alpha$ given by \eqref{eq:m_alpha2}, we have 
\be\label{eq:7.18}
M(z)= -\frac{D}{m_\alpha(z)}(1+\oo(1))
\ee 
as $ z\to \infty$ uniformly in any closed sector in $\C_+$. 
\end{enumerate}
\end{theorem}

\begin{proof}
Combining \eqref{eq:M=-m} with Corollary \ref{cor:3.08}, we conclude that \eqref{eq:7.18} is equivalent to
\begin{align*}
\x(s) & = \frac{s^{1+\alpha}}{D^{2+\alpha}}(1+\oo(1)), & \int_0^{\x(s)} \Wr(t)dt & = \oo(s^{\frac{2+\alpha}{2}})
\end{align*}
as $s\to 0$. Applying Lemma \ref{lem:a.asymp} and using a change of variables as in the proof of Theorem \ref{th:7.1}, we arrive at \eqref{eq:7.19}. 
\end{proof}

\appendix

\section{Regularly varying functions}\label{app:osv}

Let us recall the concept of regularly varying functions (see, e.g., \cite{bgt}, \cite{Kor04}).

\begin{definition}\label{def:f_rv}
Let $f:(a,\infty)\to[0,\infty)$ be measurable and eventually positive.
The function $f$ is called slowly varying at $\infty$ if
\be
\lim_{x\to \infty}\frac{f(xt)}{f(x)}=1,\quad t>0.
\ee
The function $f$ is called regularly varying  at $\infty$ with index $\alpha\in\R$ (we shall write $f\in RV_{\alpha}(\infty)$) if 
\be\label{eq:lo}
\lim_{x\to \infty}\frac{f(xt)}{f(x)}=t^\alpha, \quad t>0.
\ee
If the limit in \eqref{eq:lo} equals $\infty$ for all $t>1$, then $f$ is called rapidly varying at $\infty$ (we shall write $f\in RV_{\infty}(\infty)$).

The function $f:(0,b)\to [0,\infty)$ is called regularly (rapidly) varying at $0$ if the function $x\mapsto 1/f(1/x)$ is 
regularly (rapidly) varying at $\infty$. The corresponding notation in this case is $f\in RV_\alpha(0)$ ($f\in RV_{\infty}(0)$).
\end{definition}

Clearly, the class of slowly varying functions coincides with the class of regularly varying functions with index $0$. Note also that a regularly varying function with index $\alpha$ admits the representation $f(x)=x^\alpha \tilde{f}(x)$, where $\tilde{f}$ is a slowly varying function. Moreover, by the Karamata representation theorem \cite[Theorem IV.2.2]{Kor04}, $f$ is slowly varying at infinity precisely if there is $x_0\ge a$ such that
\be\label{eq:a.3}
f(x)=\exp\Big\{\eta(x)+\int_{x_0}^{x} \frac{\varepsilon(t)}{t}dt\Big\},\quad x\ge x_0>0,
\ee 
where $\eta$ is a bounded measurable function on $(x_0,\infty)$ such that $\lim_{x\to \infty}\eta(x)=\eta_0$, and $\varepsilon$ is a continuous function satisfying $\lim_{x\to\infty}\varepsilon(x)=0$. Regularly varying functions can be characterized by their behavior under integration against powers and this is the content of the Karamata characterization theorem (see \cite[\S IV.5]{Kor04}).

Concerning rapidly varying functions, 
it immediately follows from Definition \ref{def:f_rv} that the limit in \eqref{eq:lo} exists and equals $0$ whenever $t\in (0,1)$. Moreover, the convergence is uniform on every interval $(\lambda,\infty)$ and $(0,\lambda^{-1})$, $\lambda>1$. 
It also follows from \cite[Theorems 2.4.7 (ii) and 1.5.6]{bgt} that for every $\gamma>0$ there is a constant $C>0$ such that $f(x)\ge Cx^\gamma$ as $x\to \infty$ ($f(x)\le Cx^\gamma$ as $x\to 0$ if $f\in RV_\infty(0)$). 

We also need the following fact on the asymptotic behavior of functions and their inverses (see \cite[Chapter 1.5.7]{bgt} and \cite{dtj}).

\begin{lemma}\label{lem:a.asymp}
Let $F_0$, $F$ be two strictly increasing and positive functions on $(a,\infty)$. Let also $f_0$, $f$ denote their inverses, respectively. Suppose that $F_0$ and $F$ are regularly varying at infinity with index $\alpha\in(0,\infty]$ and $\lim_{x\to 0} F(x)/F_0(x)=1$. Then $f_0$, $f\in RV_{1/\alpha}(\infty)$ and $\lim_{y\to \infty} f(y)/f_0(y)=1$.
\end{lemma}

\bigskip
\noindent

{\bf Acknowledgments.}
We thank the referee for the careful reading of our manuscript and several critical remarks.

\end{document}